\documentclass[11pt]{article}

\usepackage{amssymb}
\usepackage{amsfonts}
\usepackage{graphicx}
\usepackage{latexsym}
\usepackage{amsmath}
\usepackage{changepage}
\usepackage{color}
\usepackage{latexsym}
\usepackage{tablefootnote}
\usepackage{epsfig}
\usepackage{multirow}
\usepackage{float}
\usepackage{array}
\usepackage{bbm}

\allowdisplaybreaks

\newtheorem{thm}{Theorem}[section]
\newtheorem{theorem}[thm]{Theorem}

\newtheorem{lemma}[thm]{Lemma}

\newtheorem{proposition}[thm]{Proposition}

\newtheorem{corollary}[thm]{Corollary}

\newtheorem{remark}[thm]{Remark}

\newtheorem{assumption}[thm]{Assumption}

\newcommand{\dt}{\odot}

\newcommand{\mb}[1]{\mbox{\boldmath $#1$}}

\begin{document}
\title{Solution to a Monotone Inclusion Problem using the Relaxed Peaceman-Rachford Splitting Method: Convergence and its Rates}

\author{{\bf{Chee-Khian Sim}}\footnote{Email address:
chee-khian.sim@port.ac.uk} \\ School of Mathematics and Physics\\
University of Portsmouth \\ Lion Gate Building, Lion Terrace \\ Portsmouth PO1 3HF}

\date{November 12, 2022}

\maketitle
\begin{abstract}
We consider the convergence behavior using the relaxed Peaceman-Rachford splitting method to solve the monotone inclusion problem $0 \in (A + B)(u)$, where $A, B: \Re^n \rightrightarrows \Re^n$ are maximal $\beta$-strongly monotone operators, $n \geq 1$ and $\beta > 0$.  Under a technical assumption, convergence of iterates using the method on the problem is proved when either $A$ or $B$ is single-valued, and the fixed relaxation parameter $\theta$ lies in the interval $(2 + \beta, 2 + \beta + \min \{ \beta, 1/\beta \})$.  With this convergence result, we address an open problem that is not settled in \cite{Monteiro} on the convergence of these iterates for $\theta \in (2 + \beta, 2 + \beta + \min \{ \beta, 1/\beta \})$.  Pointwise convergence rate results and {\textcolor{black}{$R$}}-linear convergence rate results when $\theta$ lies in the interval $[2 + \beta, 2 + \beta + \min \{\beta, 1/\beta \})$ are also provided in the paper.  Our analysis to achieve these results is atypical and hence novel.  {\textcolor{black}{Numerical experiments on the weighted Lasso minimization problem are conducted to test the validity of the assumption.}}
 
\vspace{15pt}


%\vspace{10pt}

%\noindent {\textcolor{black}{Communicated by Xinmin Yang}}
\end{abstract}

\section{Introduction}\label{sec:Intro}

\noindent We consider the following monotone inclusion problem:
\begin{eqnarray}\label{MIP}
0 \in (A + B)(u),
\end{eqnarray}
where $A, B: \Re^n \rightrightarrows \Re^n$ are point-to-set operators that are maximal $\beta$-strongly monotone. When $\beta = 0$, $A, B$ are maximal monotone operators, while when $\beta > 0$, $A, B$ are maximal and strongly monotone with constant $\beta$ in the usual sense.  In our discussion in this paper, we always consider (\ref{MIP}) having a solution and $\beta > 0$.   With $\beta > 0$, then there exists only one solution to (\ref{MIP}).  Let this unique solution be given by $u^\ast$.  
%We have the following assumption on $A, B$:
%\begin{assumption}\label{ass:single-valued}
%The maximal $\beta$-strongly monotone operator $A$ or $B$, $\beta > 0$, is single-valued.
%\end{assumption}

\vspace{10pt}

\noindent  The relaxed Peaceman-Rachford (PR) splitting method is studied extensively in the literature, such as \cite{Bauschke1}-\cite{Bauschke3}, \cite{Combettes}-\cite{Eckstein}, \cite{Giselsson}-\cite{He},  \cite{Lions,Monteiro}, to solve the monotone inclusion problem (\ref{MIP}). The method generates iterates $\{x_k\}$ using the following recursive relation: 
\begin{eqnarray}\label{RelaxedPR}
x_{k} = x_{k-1} + \theta (J_{\rho B}(2 J_{\rho A}(x_{k-1}) - x_{k-1}) - J_{\rho A}(x_{k-1})), \quad k \geq 1,
\end{eqnarray}
where $x_0$ is any point in $\Re^n$, $\theta > 0$ is a fixed relaxation parameter, $\rho > 0$ is an arbitrary scalar, and $J_T := (I + T)^{-1}$.  In this paper, we always set $\rho$ to be equal to 1.  When $\theta = 1$, the method is also known as the Douglas-Rachford splitting method, while when $\theta = 2$, it is also called the Peaceman-Rachford splitting method.  If the sequence $\{x_k\}$ generated by the relaxed PR splitting method converges to $x^\ast$, then the solution to (\ref{MIP}) is given by $u^\ast = J_A(x^\ast)$. 

\vspace{10pt}

\noindent Given $A, B$ maximal $\beta$-strongly monotone operators, and $\{x_k\}$ a sequence generated by the relaxed PR splitting method (\ref{RelaxedPR}) with $\rho = 1$, it is well-known that when $\beta = 0$, the sequence $\{ x_k \}$ generated converges for $\theta \in (0, 2)$ (see for example \cite{Bauschke, Facchinei}), while in \cite{Monteiro}, it is shown that when $\beta > 0$, the sequence $\{ x_k \}$ converges for $\theta \in (0, 2 + \beta]$.  Furthermore, in \cite{Monteiro}, an instance of (\ref{MIP}) is given for nonconvergence of $\{x_k\}$ for any $\theta \geq 2$ when $\beta = 0$.  This instance also shows nonconvergence of $\{ x_k \}$ for any $\theta \geq 2 + \beta + \min \{ \beta, 1/\beta \}$ when $\beta > 0$.  When $\beta > 0$, we see  from \cite{Monteiro} that the convergence behavior of iterates generated by (\ref{RelaxedPR}) to solve (\ref{MIP}) for $\theta \in (2 + \beta, 2 + \beta + \min \{ \beta, 1/\beta\})$ is not known, and to the best of our knowledge, has not been studied previously in the literature.  This paper attempts to fill the gap on this by investigating the convergence behavior of iterates $\{ x_k \}$ generated by (\ref{RelaxedPR}) to solve (\ref{MIP}) when $\theta \in (2 + \beta, 2 + \beta + \min \{\beta, 1/\beta\})$.  We show that an accumulation point, $x^{\ast\ast}$, of $\{ x_k \}$  has $J_A(x^{\ast\ast})$ that solves (\ref{MIP}) over this range of $\theta$, under a technical assumption.  As a consequence, if $A$ or $B$ is single-valued, then $\{ x_k \}$ converges to a limit point $x^\ast$, where $J_A(x^\ast)$ solves (\ref{MIP}). Note that for $n = 1$, not having this technical assumption leads to trivial consideration.  {\textcolor{black}{We believe that the assumption for convergence is merely technical and is not really needed for convergence to occur.  Also, through a numerical study in Section \ref{sec:Numerical}, we find that the assumption is always satisfied.}} We further believe that the convergence analysis for $\theta$ beyond $2 + \beta$ needs to be atypical {\textcolor{black}{and convergence is not shown in this range in the literature, for example, in \cite{Bartz,Dao}, where the focus of these papers is also different from ours}}.  Our analysis to achieve these results is based on finding explicit solutions to small dimensional optimization problems and small systems of inequalities as detailed in Section \ref{sec:technical}.  To add further to these contributions, we are able to provide pointwise convergence rate and {\textcolor{black}{$R$}}-linear convergence rate   results using (\ref{RelaxedPR}) to solve (\ref{MIP}) for $\theta \in [2 + \beta, 2 + \beta + \min\{\beta,1/\beta\})$, complementing results in \cite{Davis}-\cite{Eckstein}, \cite{Giselsson}-\cite{He}, \cite{Lions} and \cite{Monteiro}.  

\vspace{10pt}

\noindent This paper is divided into several sections.  In Section \ref{sec:technical}, we state and prove some technical results that are needed in a latter section, Section \ref{sec:Convergence}, to prove convergence of $\{ x_k \}$.  Section \ref{sec:preliminaries} introduces transformations on the relaxed Peaceman-Rachford (PR) splitting method (\ref{RelaxedPR}) that prepares us for analysis in Sections \ref{sec:Convergence} and \ref{sec:convergencerate}.  Section \ref{sec:Convergence} is on convergence of $\{x_k\}$, while Section \ref{sec:convergencerate} investigates pointwise convergence rate and {\textcolor{black}{$R$}}-linear convergence rate of $\{x_k\}$.  Finally, Section \ref{sec:Numerical} provides numerical results using (\ref{RelaxedPR}) to solve the weighted Lasso minimization problem.

\subsection{Conventions and Notations}\label{subsec:Notations}

\noindent $0 \cdot \infty$ is defined and is a real number, not necessarily zero.  What value it takes is dependent on the context.

\vspace{10pt}

\noindent Given $x = (x_1, \ldots, x_n)^T, y = (y_1, \ldots, y_n)^T \in (\Re \cup \{ \infty \})^n$, $x \dt y := (x_1 y_1, \ldots, x_n y_n)^T$.

\vspace{10pt}

\noindent Given $\alpha = (\alpha_1, \ldots, \alpha_n)^T \in (\Re \cup \{ \infty \})^n$, $\| \alpha \|_\infty := \sup \{ | \alpha_i |\ ; \ \alpha_i \in \Re, i = 1, \ldots, n \}$.

\vspace{10pt}

\noindent Given $x \in \Re^n$, $\| x \|$ stands for the 2-norm of $x$, while $\| x \|_1$ stands for the 1-norm of $x$.

\vspace{10pt}

\noindent {\textcolor{black}{$e \in \Re^n$ is the vector in $\Re^n$ of all ones.}

%\noindent Given $A \subseteq \{k \ ; \ k \geq 0\}$, then $A^c$ is the complement set of $A$ in $\{ k \ ; \ k \geq 0 \}$.

\section{Technical Results}\label{sec:technical}

\begin{proposition}\label{prop:2}
Let $0 < \epsilon < \min \{\beta, 1/\beta \}$, and $x, y \in \Re$.  We have
\begin{eqnarray*}
\max_{xy \geq x^2} \left[ \frac{1 -\beta}{1 + \beta} x^2 + \frac{3\beta - \epsilon - 2}{2 + \beta + \epsilon} xy + \frac{2(1 + \beta)(\epsilon - \beta)}{(2 + \beta + \epsilon)^2} y^2 \right] = 0,
\end{eqnarray*}
and is attained when $x = y = 0$.
\end{proposition}
\begin{proof}
If $x = 0$, then it is clear that the objective function in the above maximization problem is less than or equal to zero, and is equal to zero when $y$ is also equal to zero, since $\epsilon < \min \{ \beta, 1/\beta\} \leq \beta$.  Hence, we can assume that $x \not= 0$.  The proposition is then proved if we can show that the following holds:
\begin{eqnarray*}
\max_{z \geq 1} \left[ \frac{1 -\beta}{1 + \beta} + \frac{3\beta - \epsilon - 2}{2 + \beta + \epsilon}z + \frac{2(1 + \beta)(\epsilon - \beta)}{(2 + \beta + \epsilon)^2} z^2 \right] < 0.
\end{eqnarray*}
Let
\begin{eqnarray*}
f(z) := \frac{1 -\beta}{1 + \beta} + \frac{3\beta - \epsilon - 2}{2 + \beta + \epsilon}z + \frac{2(1 + \beta)(\epsilon - \beta)}{(2 + \beta + \epsilon)^2} z^2.
\end{eqnarray*}
The maximum of $f$ over $\Re$ is attained at $\frac{(2+\beta+\epsilon)(3\beta - \epsilon - 2)}{4(1 + \beta) (\beta - \epsilon)}$, which can be checked easily to be less than $1$, since $\epsilon < \min\{ \beta, 1/\beta \}$.  Hence, as $f$ is concave over $\Re$, we have
\begin{eqnarray*}
\max_{z \geq 1} \left[ \frac{1 -\beta}{1 + \beta} + \frac{3\beta - \epsilon - 2}{2 + \beta + \epsilon}z + \frac{2(1 + \beta)(\epsilon - \beta)}{(2 + \beta + \epsilon)^2} z^2 \right] = \frac{1 -\beta}{1 + \beta} + \frac{3\beta - \epsilon - 2}{2 + \beta + \epsilon} + \frac{2(1 + \beta)(\epsilon - \beta)}{(2 + \beta + \epsilon)^2}.
\end{eqnarray*}
The latter is less than zero as $0 < \epsilon < \min \{\beta, 1/\beta \}$.
\end{proof}

\begin{proposition}\label{prop:1}
Let $0 < \epsilon < \min \{ \beta, 1/\beta \}$, and $\alpha, x, y \in \Re$, where $| y | \leq 1$, be such that
\begin{eqnarray}
xy & \geq & 1, \label{ineq:barA01D} \\
\displaystyle -\left(1+\frac{2\beta( \alpha -1)}{2 + \beta + \epsilon}\right) xy - \left(\frac{(1+ \beta)(\alpha - 1)}{2 + \beta + \epsilon}\right) \left( 1 + \frac{(1 + \beta)(\alpha - 1)}{2 + \beta + \epsilon}\right)x^2  + \frac{1 - \beta}{1 + \beta} & \geq & 0, \label{ineq:barB01D}
\end{eqnarray}
then $| \alpha | \leq 1$.
\end{proposition}
\begin{proof}
\noindent Rearranging the left-hand side of (\ref{ineq:barB01D}), and letting $z = 1 - \alpha$, the left-hand side of (\ref{ineq:barB01D}) is given by the following function:
\begin{eqnarray*}
g(z) := -\frac{(1 + \beta)^2 x^2}{(2 + \beta + \epsilon)^2} z^2 + \frac{2 \beta xy + (1 + \beta)x^2}{2 + \beta + \epsilon} z - xy + \frac{1 - \beta}{1 + \beta}.
\end{eqnarray*}
It is easy to check that the above quadratic function of $z$ has its maximum point to be
\begin{eqnarray*}
z^\ast = \frac{(2 + \beta + \epsilon)(2 \beta y + (1 + \beta)x)}{2 (1 + \beta)^2 x}.
\end{eqnarray*}
Since $|y| \leq 1$, $xy \geq 1$ and $0 < \epsilon < \min \{\beta, 1/\beta\} \leq 1$, we can check that
\begin{eqnarray*}
0 < z^\ast \leq 2.
\end{eqnarray*}
For $z < 0$, since $xy \geq 1$, we have
\begin{eqnarray*}
g(z) < g(0) = - xy + \frac{1 - \beta}{1 + \beta} \leq 0.
\end{eqnarray*}
For $z > 2$, we have
\begin{eqnarray*}
g(z) < g(2) & = & \frac{2(1+ \beta)(\epsilon - \beta)}{(2 + \beta + \epsilon)^2} x^2 + \frac{3 \beta - \epsilon - 2}{2 + \beta + \epsilon} xy + \frac{1 - \beta}{1 + \beta} \\
& = & x^2 \left[ \frac{2(1+ \beta)(\epsilon - \beta)}{(2 + \beta + \epsilon)^2} + \frac{3 \beta - \epsilon - 2}{2 + \beta + \epsilon} \left(\frac{y}{x}\right) + \frac{1 - \beta}{1 + \beta} \left(\frac{1}{x^2}\right) \right] \\
& \leq & x^2 \left[ \frac{2(1+ \beta)(\epsilon - \beta)}{(2 + \beta + \epsilon)^2}y^2 + \frac{3 \beta - \epsilon - 2}{2 + \beta + \epsilon} \left(\frac{y}{x}\right) + \frac{1 - \beta}{1 + \beta} \left(\frac{1}{x^2}\right) \right] \leq 0,
\end{eqnarray*} 
where the second inequality holds since $|y | \leq 1$ and $\epsilon < \beta$, while the third inequality follows from $xy \geq 1$ and Proposition \ref{prop:2}.  Therefore, when $xy \geq 1$, for (\ref{ineq:barB01D}) to hold, that is, for $g(z) \geq 0$, we must have $0 \leq z \leq 2$, which is equivalent to $| \alpha | \leq 1$.  
\end{proof}

\begin{proposition}\label{prop:3}
Let $x, z, x_1 \in \Re_+$ and $y, y_1 \in \Re$ satisfy
\begin{eqnarray*}
y_1 - \frac{1-\beta}{1+\beta}x_1 - \frac{1 - \beta}{1 + \beta}x - \frac{3 \beta - \epsilon - 2}{2 + \beta + \epsilon} y + \frac{2(1 + \beta)(\beta - \epsilon)}{(2 + \beta + \epsilon)^2} z & \leq & 0, \\
x - y + x_1 - y_1 & \leq & 0, \\
y^2 - xz & \leq & 0, 
\end{eqnarray*}
where $0 < \epsilon < \min \{ \beta, 1/\beta \}$.  Then $x = y = z = x_1 = y_1 = 0$.
\end{proposition}
\begin{proof}
Consider the following minimization problem:
\begin{eqnarray}\label{prob:minimization1D}
\min_{x, z, x_1 \in \Re_+, y, y_1 \in \Re} y_1 - \frac{1-\beta}{1+\beta}x_1 - \frac{1 - \beta}{1 + \beta}x - \frac{3 \beta - \epsilon - 2}{2 + \beta + \epsilon} y + \frac{2(1 + \beta)(\beta - \epsilon)}{(2 + \beta + \epsilon)^2} z
\end{eqnarray}
subject to
\begin{eqnarray}
& & x - y + x_1 - y_1 \leq 0, \label{prob:constraint1} \\
& & y^2 - xz \leq 0. \label{prob:constraint2}
\end{eqnarray}
\noindent Let $(x^\ast, y^\ast, z^\ast, x_1^\ast, y_1^\ast)$ be an optimal solution to the above minimization problem (\ref{prob:minimization1D})-(\ref{prob:constraint2}). By the Fritz-John condition \cite{John} (see also \cite{Mangasarian}), there exist $\lambda_i \in \Re, i = 0, \ldots, 5$, such that the following holds:
\begin{eqnarray}
\lambda_0 
\left( \begin{array}{c}
- \frac{1 - \beta}{1 + \beta} \\
- \frac{3 \beta - \epsilon - 2}{2 + \beta + \epsilon} \\
\frac{2(1 + \beta)(\beta - \epsilon)}{(2 + \beta + \epsilon)^2} \\
- \frac{1 - \beta}{1 + \beta} \\
1
\end{array} \right) +
\lambda_1 
\left( \begin{array}{r}
1 \\
-1 \\
0 \\
1 \\
-1 
\end{array} \right) + 
\lambda_2
\left( \begin{array}{c}
-z^\ast \\
2y^\ast \\
-x^\ast \\
0 \\
0 
\end{array} \right) +
\lambda_3
\left( \begin{array}{r}
-1 \\
0 \\
0 \\
0 \\
0
\end{array} \right) +
\lambda_4
\left( \begin{array}{r}
0 \\
0 \\
-1 \\
0 \\
0
\end{array} \right) + 
\lambda_5
\left( \begin{array}{r}
0 \\
0 \\
0 \\
-1 \\
0
\end{array} \right)  = 0, \label{cond:FJ0}
\end{eqnarray}
with
\begin{eqnarray}
& & \lambda_0 \geq 0, \quad (\lambda_0, \ldots, \lambda_5) \not= 0, \label{cond:FJ1} \\
& &\lambda_1(x^\ast - y^\ast + x_1^\ast - y_1^\ast) = 0, \quad \lambda_1 \geq 0, \quad x^\ast - y^\ast + x_1^\ast - y_1^\ast \leq 0, \label{cond:FJ2} \\
& & \lambda_2((y^\ast)^2 - x^\ast z^\ast) = 0, \quad \lambda_2 \geq 0, \quad (y^\ast)^2 - x^\ast z^\ast \leq 0, \label{cond:FJ3} \\
& & \lambda_3 x^\ast = 0, \quad \lambda_3 \geq 0, \quad x^\ast \geq 0, \label{cond:FJ4} \\
& & \lambda_4 z^\ast = 0, \quad \lambda_4 \geq 0, \quad z^\ast \geq 0, \label{cond:FJ5}\\
& & \lambda_5 x_1^\ast = 0 ,\quad \lambda_5 \geq 0, \quad x_1^\ast \geq 0. \label{cond:FJ6}
% \\
%& & \lambda_6 (x^\ast - y^\ast) = 0, \quad \lambda_6 \geq 0, \quad x^\ast - y^\ast \leq 0.  \label{cond:FJ7}
\end{eqnarray}

\noindent {\bf{Case $\mb{1}$: $\mb{\lambda_0 \not= 0}$.}}

\vspace{5pt}

\noindent We show that this case is impossible from (\ref{cond:FJ0})-(\ref{cond:FJ6}).  WLOG, let $\lambda_0 = 1$.  Then, from the last ``row" in (\ref{cond:FJ0}), we get that $\lambda_1 = 1$.  From the first to the third ``row" in (\ref{cond:FJ0}) and $\lambda_1 = 1$, we get
\begin{eqnarray}
\lambda_2 z^\ast + \lambda_3 & = & \frac{2 \beta}{1 + \beta}, \label{eq:FJ0_1} \\
\lambda_2 y^\ast  & = & \frac{2 \beta}{2 + \beta + \epsilon}, \label{eq:FJ0_2} \\
\lambda_2 x^\ast + \lambda_4 & = & \frac{2 (1 + \beta)(\beta - \epsilon)}{(2 + \beta + \epsilon)^2}. \label{eq:FJ0_3}
\end{eqnarray}
Since the right-hand side of (\ref{eq:FJ0_2}) is positive, we have $\lambda_2 \not= 0, y^\ast \not= 0$.  Hence, from the equality in (\ref{cond:FJ3}), we get that $(y^\ast)^2 = x^\ast z^\ast$, $x^\ast \not= 0$.  Since $x^\ast \not= 0$, we have from the equality in (\ref{cond:FJ4}) that $\lambda_3 = 0$.  Hence, (\ref{eq:FJ0_1}) becomes
\begin{eqnarray}\label{eq:FJ0_1prime}
\lambda_2 z^\ast  =  \frac{2 \beta}{1 + \beta}.
\end{eqnarray}
On the other hand, multiplying both sides of the equality in (\ref{eq:FJ0_3}) by $z^\ast$, using $(y^\ast)^2 = x^\ast z^\ast$ and the equality in (\ref{cond:FJ5}), we get
\begin{eqnarray*}
\lambda_2 (y^\ast)^2 = \frac{2 (1 + \beta)(\beta - \epsilon)}{(2 + \beta + \epsilon)^2} z^\ast.
\end{eqnarray*}
Hence,
\begin{eqnarray}\label{eq:FJ0_3prime}
z^\ast = \frac{({\textcolor{black}{2}} + \beta + \epsilon)^2}{2(1 + \beta)(\beta - \epsilon)} \lambda_2 (y^\ast)^2.
\end{eqnarray}
Substituting the expression for $z^\ast$ in (\ref{eq:FJ0_3prime}) into (\ref{eq:FJ0_1prime}), we can solve for $\lambda_2 y^\ast$ to be
\begin{eqnarray*}
\lambda_2 y^\ast = \pm \frac{2 \sqrt{\beta(\beta - \epsilon)}}{2 + \beta + \epsilon}.
\end{eqnarray*}
Comparing the above expression for $\lambda_2 y^\ast$ with that in (\ref{eq:FJ0_2}) leads to a contradiction.  Hence, $\lambda_0$ cannot be nonzero.
%Solving the above quadratic equation for $\lambda_2 y^\ast$, we have
%\begin{eqnarray*}
%\lambda_2 y^\ast = \frac{2 (\beta - \epsilon)}{(2 + \beta + \epsilon)^2}[1 + \beta \pm \sqrt{1 - \beta \epsilon}].
%\end{eqnarray*}
%Substituting the above expression for $\lambda_2 y^\ast$ into (\ref{eq:FJ0_2}), we get that
%\begin{eqnarray*}
%\lambda_6 & = & \frac{4(\beta - \epsilon)}{(2 + \beta + \epsilon)^2}[1 + \beta \pm \sqrt{1 - \beta \epsilon}] - \frac{4 \beta}{2 + \beta + \epsilon} \\
%& = & \frac{4}{(2 + \beta + \epsilon)^2}[(\beta - \epsilon)(1 + \beta \pm \sqrt{1 - \beta \epsilon}) - \beta(2 + \beta + \epsilon)] \\
%& = & \frac{4}{(2 + \beta + \epsilon)^2}[ - \beta(1 + \epsilon) - \epsilon(1 + \beta) \pm (\beta - \epsilon)\sqrt{1 - \beta \epsilon}] < 0,
%\end{eqnarray*}
%which is a contradiction to $\lambda_6 \geq 0$.

\vspace{5pt}

\noindent {\bf{Case 2: $\mb{\lambda_0 = 0}$.}}

\vspace{5pt}

\noindent In this case, (\ref{cond:FJ0}) becomes
\begin{eqnarray}
\lambda_1 
\left( \begin{array}{r}
1 \\
-1 \\
0 \\
1 \\
-1 
\end{array} \right) + 
\lambda_2
\left( \begin{array}{c}
-z^\ast \\
2y^\ast \\
-x^\ast \\
0 \\
0 
\end{array} \right) +
\lambda_3
\left( \begin{array}{r}
-1 \\
0 \\
0 \\
0 \\
0
\end{array} \right) +
\lambda_4
\left( \begin{array}{r}
0 \\
0 \\
-1 \\
0 \\
0
\end{array} \right) + 
\lambda_5
\left( \begin{array}{r}
0 \\
0 \\
0 \\
-1 \\
0
\end{array} \right) = 0. \label{cond:FJ0prime}
\end{eqnarray}
We immediately observe from (\ref{cond:FJ0prime}) that $\lambda_1 = \lambda_5 = 0$.  Also note that not all $\lambda_i = 0$ is equal to zero, $i = 2, 3$ and $4$.

\vspace{5pt}

\noindent If $\lambda_2 \not= 0$. Then from the third ``row" in (\ref{cond:FJ0prime}), $x^\ast \geq 0$ and $\lambda_4 \geq 0$, we have that $x^\ast = \lambda_4 = 0$.  It then follows from the last inequality in (\ref{cond:FJ3}) that $y^\ast = 0$.  Furthermore, from the first ``row" in (\ref{cond:FJ0prime}), we have $z^\ast = 0$ since $z^\ast \geq 0$ and $\lambda_3 \geq 0$.

\vspace{5pt}

\noindent Note that $\lambda_3$ has to be zero since if it is positive, then this leads to a contradiction in the first ``row" of (\ref{cond:FJ0prime}) as $z^\ast \geq 0$ and $\lambda_2 \geq 0$.  

\vspace{5pt}

\noindent If $\lambda_4 \not= 0$.  Then from the third ``row" in (\ref{cond:FJ0prime}), we have $-\lambda_2 x^\ast = \lambda_4 > 0$, which is impossible.  Hence, $\lambda_4 = 0$.

%\vspace{5pt}
%
%\noindent If $\lambda_6 \not= 0$.  Then from the first ``row" in (\ref{cond:FJ0prime}) and $\lambda_3 = 0$ that $\lambda_2 \not= 0$, which leads to a contradiction by an argument above.  Hence $\lambda_6 = 0$.

\vspace{5pt}

\noindent Hence, in this case, we have $x^\ast = y^\ast = z^\ast = 0$, with $\lambda_2 > 0$ and $\lambda_1 = \lambda_3 = \lambda_4 = \lambda_5  = 0$.  Finally, since $x_1^\ast \leq y_1^\ast$, we observe that the optimal value to the minimization problem (\ref{prob:minimization1D})-(\ref{prob:constraint2}) is greater than or equal to $\frac{2\beta}{1 + \beta} x_1^\ast (\geq 0)$.  In fact, it is equal to zero with $x_1^\ast = y_1^\ast = 0$.

\vspace{5pt}

\noindent In conclusion, the optimal value of the minimization problem (\ref{prob:minimization1D})-(\ref{prob:constraint2}) is zero with optimal solution $x^\ast = y^\ast = z^\ast = x_1^\ast = y_1^\ast = 0$.  The proposition then follows.
\end{proof}

\section{Preliminaries to Convergence Analysis}\label{sec:preliminaries}

\noindent We observe a few facts in this paragraph.  Recall that if $\{x_k\}$ generated by the relaxed PR splitting method (\ref{RelaxedPR}) converges to $x^\ast$ for a given $\theta > 0$, and if we let $J_A(x^\ast) = u^\ast$, then $u^\ast \in \Re^n$ is a solution to (\ref{MIP}), which is unique.  Hence, there exists $z^\ast \in \Re^n$ such that $z^\ast \in A(u^\ast), -z^\ast \in B(u^\ast)$, since $u^\ast$ is a solution to (\ref{MIP}).  Note that $z^\ast$ that satisfies $z^\ast \in A(u^\ast), -z^\ast \in B(u^\ast)$ is unique if either $A$ or $B$ is single-valued.  It is also easy to see in this case that $x^\ast = u^\ast + z^\ast$.  We assume\footnote{Note that because of this assumption on $u^\ast$, we have $x^\ast = z^\ast$, when $A$ or $B$ is single-valued.} $u^\ast = 0$ without loss of generality from now onwards.  We can do this because by letting $A_1(\cdot) = A(\cdot + u^\ast)$ and $B_1(\cdot) = B(\cdot + u^\ast)$, we observe that $z^\ast \in A_1(0), -z^\ast \in B_1(0)$, and the relaxed PR splitting method (\ref{RelaxedPR}) using $(A,B)$ and using $(A_1,B_1)$ generate sequence with corresponding terms in each sequence differing from each other by $u^\ast$.  

%Finally, we conclude this paragraph based on our discussions by stating that if 
%\begin{eqnarray}\label{uniqueness0}
%J_A(x^{\ast\ast}) = u^{\ast \ast}, \quad J_B(2u^{\ast\ast} - x^{\ast\ast}) = u^{\ast\ast},
%\end{eqnarray}
%then $x^{\ast\ast} = x^\ast$, $u^{\ast\ast} = 0$. 

\vspace{10pt}

\noindent Let $A_0 := A - \beta I$ and $B_0 := B - \beta I$.  Then, $A_0$ and $B_0$ are maximal monotone operators from $\Re^n$ to $\Re^n$.  In terms of $A_0, B_0$, (\ref{RelaxedPR}) is given by
\begin{eqnarray}\label{method:modifiedRelaxePR}
\bar{x}_k = \bar{x}_{k-1} + \frac{\theta}{1 + \beta}\left(J_{\frac{1}{1+ \beta}B_0} \left(\frac{2}{1+ \beta}J_{\frac{1}{1+\beta}A_0}(\bar{x}_{k-1}) - \bar{x}_{k-1} \right) - J_{\frac{1}{1+\beta}A_0}(\bar{x}_{k-1}) \right), \quad k \geq 1,
\end{eqnarray}
where $\bar{x}_k := \frac{1}{1+\beta}x_k$.  Letting $\bar{A}_0 := \frac{1}{1+\beta}A_0$ and $\bar{B}_0 := \frac{1}{1+\beta}B_0$, we can rewrite (\ref{method:modifiedRelaxePR}) as
\begin{eqnarray}\label{method:modifiedRelaxedPR2}
\bar{x}_k = \bar{x}_{k-1} + \frac{\theta}{1 + \beta}\left(J_{\bar{B}_0} \left(\frac{2}{1+ \beta}J_{\bar{A}_0}(\bar{x}_{k-1}) - \bar{x}_{k-1} \right) - J_{\bar{A}_0}(\bar{x}_{k-1}) \right), \quad k \geq 1.
\end{eqnarray}
Note that $\bar{A}_0$ and $\bar{B}_0$ are maximal monotone operators, and that $\{x_k\}$ converges to $x^\ast$ if and only if $\{ \bar{x}_k\}$ converges to $\bar{x}^\ast := \frac{1}{1 + \beta} x^\ast$. Furthermore, we have $J_{\bar{A}_0}(\bar{x}^\ast) = 0, J_{\bar{B}_0}(-\bar{x}^\ast) = 0$, that is,
\begin{eqnarray}\label{inc:barA0barB0}
\bar{x}^\ast \in \bar{A}_0(0),\quad -\bar{x}^\ast \in \bar{B}_0(0).
\end{eqnarray}
Note that suppose $A$ or $B$ is single-valued, then if there exists $\bar{x}^{\ast\ast} \in \Re^n$ that satisfies (\ref{inc:barA0barB0}), we have $\bar{x}^{\ast \ast} = \bar{x}^\ast$.
%We point out that in this note, even though either $\bar{A}_0$ or $\bar{B}_0$ is a single-valued function, we still use the inclusion notation as in (\ref{barA0barB0}) above, instead of equality, for the sake of convenience.  It is worthwhile to note based on (\ref{uniqueness0}) and its conclusion that if there exist $\bar{x}^{\ast \ast}, \bar{u}^{\ast \ast}$ such that
%\begin{eqnarray}\label{uniqueness}
%J_{\bar{A}_0}(\bar{x}^{\ast\ast}) = \bar{u}^{\ast\ast}, \quad J_{\bar{B}_0}\left( \frac{2}{1+\beta} \bar{u}^{\ast\ast} - \bar{x}^{\ast\ast} \right) = \bar{u}^{\ast\ast},
%\end{eqnarray}
%then $\bar{x}^{\ast\ast} = \bar{x}^\ast$, $\bar{u}^{\ast\ast} = 0$.  

\vspace{10pt}

\noindent We consider $\theta = 2 + \beta + \epsilon$, where $0 \leq \epsilon < \min \{ \beta, 1/\beta \}$, $\beta > 0$, in this paper.

\vspace{10pt}

\noindent Let $\bar{u}_k := J_{\bar{A}_0}(\bar{x}_k)$, $\bar{v}_k := J_{\bar{B}_0}\left( \frac{2}{1+ \beta} \bar{u}_k - \bar{x}_k \right)$.
%, and $\bar{x}_{k-1}^i = \bar{x}^\ast_i$ implies that $\alpha_{k-1}^i$ is either $0$ or $\infty$.)  
Hence,
\begin{eqnarray}
\bar{x}_{k} - \bar{u}_{k} & \in & \bar{A}_0(\bar{u}_{k}), \label{inc:barA0prime} \\
\bar{w}_{k} & \in & \bar{B}_0(\bar{v}_{k}), \label{inc:barB0prime}
\end{eqnarray}
where
\begin{eqnarray}
\bar{w}_{k} := \frac{2}{1 + \beta} \bar{u}_{k}  - \bar{x}_{k} - \bar{v}_{k}. \label{def:barwk0} 
\end{eqnarray}
By monotonicity of $\bar{A}_0$ and $\bar{B}_0$, for $k \geq 0$, we have by  (\ref{inc:barA0barB0}) - (\ref{inc:barB0prime}) that
\begin{eqnarray}
\langle (\bar{x}_{k} - \bar{x}^\ast) - \bar{u}_{k}, \bar{u}_{k} \rangle & \geq & 0, \label{ineq:barA0}\\
\langle \bar{w}_{k} + \bar{x}^\ast, \bar{v}_{k} \rangle  & \geq & 0. \label{ineq:barB0}
\end{eqnarray}
\noindent Observe from (\ref{ineq:barA0}) that
\begin{eqnarray}\label{ineq:barA0prime}
\langle \bar{x}_{k} - \bar{x}^\ast, \bar{u}_{k} \rangle \geq \| \bar{u}_{k} \|^2.
\end{eqnarray}
We can also write (\ref{ineq:barB0}), using (\ref{def:barwk0}), as
\begin{eqnarray}\label{ineq:barB0alternative}
\langle \bar{x}_{k} - \bar{x}^\ast, \bar{v}_{k} \rangle \leq \frac{2}{1+\beta} \langle \bar{u}_{k}, \bar{v}_{k} \rangle  - \| \bar{v}_{k} \|^2.
\end{eqnarray}
Inequalities (\ref{ineq:barA0prime}) and (\ref{ineq:barB0alternative}) play important roles to arrive at the convergence and convergence rates results in this paper.

\vspace{10pt}

\noindent We end this section with the following, which we state without proof:
\begin{proposition}\label{prop:simpleobservation}
For all $k \geq 1$, we have $x_{k} - x_{k-1} = (1 + \beta)(\bar{x}_k - \bar{x}_{k-1}) = \theta (\bar{v}_{k-1} - \bar{u}_{k-1})$.
\end{proposition}

%\noindent From now onwards and into later sections, any reference to $\bar{x}_k$, $\bar{u}_k$ is with the understanding that they satisfy (\ref{ineq:barA0prime}), (\ref{ineq:barB0alternative}) (or  (\ref{ineq:barB0prime})), where $\bar{x}_k = \alpha_{k-1}\dt(\bar{x}_{k-1} - \bar{x}^\ast) + \bar{x}^\ast$, $\bar{u}_k = J_{\bar{A}_0}(\bar{x}_k)$ and $\bar{v}_k = J_{\bar{B}_0}\left(\frac{2}{1+\beta}\bar{u}_k - \bar{x}_k \right)$. 
%
%\vspace{10pt}

%\noindent The following is a simple observation, which we state without proof:

%\vspace{10pt}
%
%\noindent For $i = 1, \ldots, n$, define $S^i := \{ k \geq 0 \ ; \ \alpha_k^i = 0\ {\rm{and}}\ \bar{x}_k^i \not= \bar{x}^\ast_i\}$, $S_1^i := \{ k \geq 0 \ ; \ \alpha_k^i = \infty \}$ and $S_2^i := (S^i \cup S_1^i)^c$. 
%Further define ${\alpha}_j$, $j \geq 0$, to be such that if $j \in S^i \cup S_1^i$, then ${\alpha}_j^i := \alpha_j^i$, and if $j \in S_2^i$, then ${\alpha}_j^i$ is such that for $j_1 \in S_2^i$, where $j_1$ is enumerated just larger than $j$ in $S_2^i$, we have $\bar{x}_{j_1}^i - \bar{x}^\ast_i = {\alpha}^i_j(\bar{x}_{j}^i - \bar{x}^\ast_i)$.  If no such $j_1$ exists, then let ${\alpha}^i_j = \alpha^i_j$. 

\section{Convergence of $\mb{\{x_k\}}$}\label{sec:Convergence}

\noindent We begin the section by stating the convention that the $i^{th}$ component of $\bar{x}_k, \bar{u}_k, \bar{v}_k$ is to be written as $\bar{x}_k^i, \bar{u}_k^i$ and $\bar{v}_k^i$ respectively, while the $i^{th}$ component of $\bar{x}^\ast$ is denoted by $\bar{x}^\ast_i$.  We use the same convention for the $i^{th}$ component of $x_k, u_k, v_k$ and $x^\ast$.  Using this convention, we state a technical assumption on $\bar{x}_k$ that is to apply to the whole  section:
\begin{assumption}\label{ass:Si}
For all $i = 1, \ldots, n$ and for all $k \geq K_0$, where $K_0 \geq 0$, we have $\bar{x}_k^i \not= \bar{x}^\ast_i$.
\end{assumption}
The above assumption makes the analysis in this section possible.

\begin{remark}\label{n=1}
For $n = 1$, if Assumption \ref{ass:Si} does not hold, that is, there exists $k_0 \geq K_0$ such that $\bar{x}_{k_0} = \bar{x}^\ast$.  Then it is easy to see that for all $k \geq k_0$, we have $\bar{x}_k = \bar{x}^\ast$.  Hence, not having Assumption \ref{ass:Si} when $n = 1$ leads to a trivial situation.
\end{remark}  

%\vspace{10pt}

\noindent For $k \geq 1$, let us write $\bar{x}_k$ as $\alpha_{k-1} \dt(\bar{x}_{k-1} - \bar{x}^\ast) + \bar{x}^\ast$, where $\alpha_{k-1} \in (\Re \cup \{ \infty \})^n$  (with the understanding that $\alpha_{k-1}^i = \infty$ if and only if $\bar{x}_{k-1}^i = \bar{x}^\ast_i$.  Here, $\alpha_{k-1}^i$ is the $i^{th}$ component of $\alpha_{k-1}$.)  From (\ref{method:modifiedRelaxedPR2}) and with $\theta = 2 + \beta + \epsilon$, $\bar{v}_k$ defined by $J_{\bar{B}_0}\left( \frac{2}{1+\beta}\bar{u}_k - \bar{x}_k \right)$, where $\bar{u}_k = J_{\bar{A}_0}(\bar{x}_k)$, can then be written as
\begin{eqnarray}\label{def:barvk}
\bar{v}_{k} =  \frac{1 + \beta}{2 + \beta + \epsilon}(\alpha_{k} - e)\dt(\bar{x}_{k} - \bar{x}^\ast) + \bar{u}_{k},  
\end{eqnarray}
and, by (\ref{def:barvk}), $\bar{w}_k$ defined by (\ref{def:barwk0}) can be written as
\begin{eqnarray}\label{def:barwk}
\bar{w}_k &  = & \frac{1-\beta}{1+\beta} \bar{u}_{k} - \bar{x}_{k} - \frac{1 + \beta}{2 + \beta + \epsilon}(\alpha_{k} - e)\dt(\bar{x}_{k} - \bar{x}^\ast). 
\end{eqnarray}
Using (\ref{def:barvk}),   (\ref{def:barwk}), we can write (\ref{ineq:barB0}) in the following way:
\begin{eqnarray}
\sum_{i=1}^n \left[ - \left(1 + \frac{2\beta(\alpha_{k}^i -1)}{2 + \beta + \epsilon}\right) \bar{u}_{k}^i(\bar{x}_{k}^i - \bar{x}^\ast_i) \right. & & \nonumber \\ 
\left. - \left(\frac{(1+ \beta)(\alpha_{k}^i - 1)}{2 + \beta + \epsilon} \right)\left(1 + \frac{(1 + \beta)(\alpha_{k}^i - 1)}{2 + \beta + \epsilon}\right)(\bar{x}_{k}^i - \bar{x}^\ast_i)^2 \right] + \frac{1 - \beta}{1 + \beta} \| \bar{u}_{k} \|^2 & \geq & 0, \nonumber \\
& & \label{ineq:barB0prime}
\end{eqnarray}
where $0 < \epsilon < \min \{ \beta, 1/\beta \}$.

%\vspace{10pt}
%
%\noindent sumption \ref{ass:Si} is assumed to hold throughout this paper.

\begin{remark}\label{rem:simplifyingassumption}
With $\bar{x}_k$ given by $\alpha_{k-1} \dt(\bar{x}_{k-1} - \bar{x}^\ast) + \bar{x}^\ast$, we deduce from Assumption \ref{ass:Si} that for all $i = 1, \ldots, n$ and for all $k \geq K_0$, we have $\alpha_k^i \not= 0$ or $\infty$.  WLOG, we let $K_0$ in Assumption \ref{ass:Si} to be equal to zero from now onwards. When $n \geq 2$, if we allow $\alpha_k^i = 0$ or $\infty$ for some $i = 1, \ldots, n$, and for infinity many $k$, then our approach to proving convergence of iterates generated by the relaxed PR splitting method (\ref{RelaxedPR}) for $\theta = 2 + \beta + \epsilon$, where $0 < \epsilon < \min \{\beta, 1/\beta\}$, $\beta > 0$, in this section, has to be further refined, and may not be able to carry through.  When $n = 1$, however, having $\alpha_k^1 = 0$ or $\infty$ for some $k \geq 0$ leads trivially to $\bar{x}_j = \bar{x}^\ast$ for all $j \geq k+1$.
\end{remark}  

%\begin{remark}\label{rem:explanation}
%Assumption \ref{ass:Si} is needed to obtain the result in Theorem \ref{thm:maintheoremnD} and the $Q$-linear convergence rate of $\| x_k - x^\ast \|$ in Theorems \ref{thm:gammalessthanzero}, \ref{thm:gammagreaterthanzero}.  It is however not required to obtain the iteration convergence rate of $\| x_k - x_{k-1} \|$ in Theorems \ref{thm:gammalessthanzero}, \ref{thm:gammagreaterthanzero}.
%\end{remark}

\noindent The next two lemmas enable us to prove Theorem \ref{thm:maintheoremnD} on the convergence behavior of $\{ x_k \}$.
\begin{lemma}\label{lem:alphakboundnD}
Given $\bar{x}_k = \alpha_{k-1} \dt (\bar{x}_{k-1} - \bar{x}^\ast) + \bar{x}^\ast$, where $\{\bar{x}_k\}$ is generated by (\ref{method:modifiedRelaxedPR2}) with $\theta = 2 + \beta + \epsilon$, $0 < \epsilon < \min \{ \beta, 1/\beta \}$, $\beta > 0$, and $\bar{x}_0 \in \Re^n$, we have $\{ \bar{x}_k \}$ is bounded.  Furthermore, for each $i = 1 ,\ldots, n$,  if $\{\bar{x}^i_{j_k}\}$ is a convergent sequence of $\{ \bar{x}_k^i \}$, then if  $\lim_{k \rightarrow \infty} \bar{x}_{j_k}^i \not= \bar{x}^\ast_i$, we have $\lim_{k \rightarrow \infty} | \alpha_k^i | = 1$.
%\begin{eqnarray*}
%\lim_{k \rightarrow \infty}{\rm{sup}}_{j \geq k} \| {\alpha}_j \|_{\infty} \leq 1.
%\end{eqnarray*}
%Consequently, 
%\begin{eqnarray*}
%\lim_{k \rightarrow \infty}{\rm{sup}}_{j \geq k} \| {\alpha}_j \|_{\infty} \leq 1.
%\end{eqnarray*}
\end{lemma}
\begin{proof}
Suppose $\{ \bar{x}_k \}$ is unbounded.  For each $i = 1, \ldots, n$, and $N \geq 1$, let $I_N^i := \{ k \geq 0\ ; \ | \bar{x}^i_k | \geq N \}$.  Since we assume $\{ \bar{x}_k \}$ is unbounded, the set $U \subseteq \{ 1, \ldots, n \}$ defined to be such that for each $i \in U$, $\bigcap_{N=1}^{\infty} I_N^i \not= \emptyset$, is nonempty.

\vspace{5pt}

\noindent We observe that in order for (\ref{ineq:barA0prime}), (\ref{ineq:barB0prime}) to hold, there must exist $i_0 \in U$ and $N_0 \geq 1$ such that for all $k \in I^{i_0}_{N_0}$, we have
\begin{eqnarray*}
(\bar{x}_{k}^{i_0} - \bar{x}^\ast_{i_0})\bar{u}_{k}^{i_0} & \geq & \| \bar{u}_{k} \|^2, \\
- \left(1 + \frac{2\beta(\alpha_{k}^{i_0} -1)}{2 + \beta + \epsilon}\right) \bar{u}_{k}^{i_0}(\bar{x}_{k}^{i_0} - \bar{x}^\ast_{i_0}) & &   \\
- \left(\frac{(1+ \beta)(\alpha_{k}^{i_0} - 1)}{2 + \beta + \epsilon} \right)\left(1 + \frac{(1 + \beta)(\alpha_{k}^{i_0} - 1)}{2 + \beta + \epsilon}\right)(\bar{x}_{k}^{i_0} - \bar{x}^\ast_{i_0})^2 + \frac{1 - \beta}{1 + \beta} \| \bar{u}_{k} \|^2 & \geq & 0.
\end{eqnarray*}
From the above two inequalities, by Proposition \ref{prop:1}, where we let $x = (\bar{x}_{k}^{i_0} - \bar{x}^\ast_{i_0})/\| \bar{u}_{k} \|, y = \bar{u}_{k}^{i_0}/ \| \bar{u}_{k} \|, \alpha = \alpha_{k}^{i_0}$ in the proposition, we have $| \alpha_{k}^{i_0} | \leq 1$ for all $k \in I^{i_0}_{N_0}$.  We see that this implies that $\{ \bar{x}_k^{i_0} \}$ is bounded.  But this is a contradiction to $i_0 \in U$, that is, $\bigcap_{N=1}^{\infty} I_N^{i_0} \not= \emptyset$.  Hence, $\{ \bar{x}_k \}$ is bounded.  

\vspace{5pt}

\noindent We now show that for $i = 1, \ldots, n$, if $\{ \bar{x}_{j_k}^i \}$ is a convergent subsequence of $\{ \bar{x}_k^i \}$ with limit point $\bar{x}^{\ast\ast}_i$, which is not equal to $\bar{x}^\ast_i$, then $\lim_{k \rightarrow \infty} | \alpha_k^i | = 1$.  We have
\begin{eqnarray*}
\bar{x}^{i}_{j_k} - \bar{x}^\ast_{i} = \Pi_{0 \leq j \leq j_k - 1} {\alpha}^{i}_j (\bar{x}^{i}_{0} - \bar{x}^{\ast}_{i}).
\end{eqnarray*}
Hence,
\begin{eqnarray*}
\log \left| \frac{\bar{x}^{i}_{j_k} - \bar{x}^\ast_{i}}{\bar{x}^{i}_{0} - \bar{x}^\ast_{i}} \right| = \sum_{0 \leq j \leq j_k - 1} \log | {\alpha}^{i}_j |,
\end{eqnarray*}
{\textcolor{black}{which holds since $\alpha_j^i \not= \infty$, due to Assumption \ref{ass:Si}.}} Since $\{ \bar{x}^{i}_{j_k} \}$ converges to $\bar{x}^{\ast\ast}_i$, which is not equal to $\bar{x}^\ast_i$, this implies that $\left\{\sum_{0 \leq j \leq j_k - 1} \log | {\alpha}^{i}_j | \right\}$ converges, which further implies that $\lim_{k \rightarrow \infty} \log | {\alpha}^{i}_{k} | = 0$.  That is, $\lim_{k \rightarrow \infty} | {\alpha}^{i}_{k}| = 1$.
%, which is a contradiction to $|{\alpha}_{j_k}^{i_0}| \geq 1 + \epsilon$ for all $k \geq K$.  Hence, we show that $\lim_{k \rightarrow \infty}{\rm{sup}}_{j \geq k} \| {\alpha}_j \|_{\infty} \leq 1$.
%
%
%
%$\lim_{k \rightarrow \infty}{\rm{sup}}_{j \geq k} \| {\alpha}_j \|_{\infty} \leq 1$ is proved.  The consequence in the lemma then follows from this and the definition of ${\alpha}_k$.
\end{proof}

\begin{lemma}\label{lem:technical2}
Given $\bar{u} = (\bar{u}_1, \ldots, \bar{u}_n)^T, \bar{x} = (\bar{x}_1, \ldots, \bar{x}_n)^T, \bar{\bar{x}} = ({\bar{\bar{x}}}_1, \ldots, {\bar{\bar{x}}}_n)^T \in \Re^n$ and $0 < \epsilon < \min \{\beta, 1/\beta \}$.  Suppose $\langle \bar{u}, \bar{\bar{x}} - \bar{x} \rangle  \geq  \| \bar{u} \|^2$ and 
\begin{eqnarray}
\langle \bar{\bar{x}} - \bar{x}, \bar{u} \rangle \leq \frac{1 - \beta}{1 + \beta} \| \bar{u} \|^2  + \sum_{i \in I} \left[ \left(\frac{4\beta}{2 + \beta + \epsilon} \right) \bar{u}_i(\bar{\bar{x}}_i  - \bar{x}_i) + \frac{2(1 + \beta)(\epsilon - \beta)}{(2 + \beta + \epsilon)^2}(\bar{\bar{x}}_i - \bar{x}_i)^2 \right],  \label{ineq:technical2}
\end{eqnarray}
where $I \subseteq \{1, \ldots, n\}$.  Then $\bar{u} = 0$, and $\bar{x}_i = \bar{\bar{x}}_i$ for $i \in I$.
\end{lemma}
\begin{proof}
From (\ref{ineq:technical2}), we have
\begin{eqnarray}
0 & \leq & - \sum_{i \not\in I} \bar{u}_i (\bar{\bar{x}}_i - \bar{x}_i) + \frac{1 - \beta}{1 + \beta} \sum_{i \not\in I} \bar{u}_i^2 + \frac{1 - \beta}{1 + \beta} \sum_{i \in I} \bar{u}_i^2 + \frac{3\beta - \epsilon - 2}{2 + \beta + \epsilon} \sum_{i \in I} \bar{u}_i(\bar{\bar{x}}_i  - \bar{x}_i) \nonumber \\
& &  + \frac{2(1 + \beta)(\epsilon - \beta)}{(2 + \beta + \epsilon)^2}\sum_{i \in I}(\bar{\bar{x}}_i - \bar{x}_i)^2. \label{ineq:technical2prime}
\end{eqnarray}
From $\langle \bar{u}, \bar{\bar{x}} - \bar{x} \rangle \geq \| \bar{u} \|^2$, we have
\begin{eqnarray}\label{ineq:technical3}
0 \leq \sum_{i \in I} \bar{u}_i ( \bar{\bar{x}}_i - \bar{x}_i) - \sum_{i \in I} \bar{u}_i^2 + \sum_{i \not\in I} \bar{u}_i( \bar{\bar{x}}_i - \bar{x}_i) - \sum_{i \not\in I} \bar{u}_i^2.
\end{eqnarray}
The lemma follows from Proposition \ref{prop:3} by letting $x = \sum_{i \in I} \bar{u}_i^2, y = \sum_{i \in I} \bar{u}_i(\bar{\bar{x}}_i - \bar{x}_i), z = \sum_{i \in I}(\bar{\bar{x}}_i - \bar{x}_i)^2, x_1 = \sum_{i \not\in I} \bar{u}_i^2$ and $y_1 = \sum_{i \not\in I} \bar{u}_i(\bar{\bar{x}}_i - \bar{x}_i)$ in the proposition.  Note that the condition $y^2 - xz \leq 0$ in the proposition holds in this case by the Cauchy-Schwartz inequality.
\end{proof}

\vspace{10pt}

\noindent We have the following theorem:
\begin{theorem}\label{thm:maintheoremnD}
For $\theta = 2 + \beta + \epsilon$, where $0 < \epsilon < \min \{ \beta, 1/\beta \}$, $\beta > 0$, let $x^{\ast\ast}$ be an accumulation point of  $\{ x_k \}$, where the sequence $\{ x_k \}$ is generated by the relaxed PR splitting method (\ref{RelaxedPR}) for a given initial iterate $x_0 \in \Re^n$.  Then $J_A( x^{\ast\ast})$ solves the monotone inclusion problem (\ref{MIP}).
\end{theorem}
\begin{proof}
We only need to prove that for $\theta = 2 + \beta + \epsilon$, where $0 < \epsilon < \min \{ \beta, 1/\beta \}$, $\beta > 0$, any convergent subsequence $\{ \bar{x}_{j_k} \}$ of $\{ \bar{x}_k \}$, the latter generated using (\ref{method:modifiedRelaxedPR2}) when $\bar{x}_0 = \frac{1}{1 + \beta} x_0$, has its limit point $\bar{x}^{\ast\ast}$ that satisfies (\ref{inc:barA0barB0}). 

%We consider two cases.
%
%\vspace{5pt}
%
%\noindent {\bf{Case 1: $\mb{\lim_{k \rightarrow \infty}{\rm{sup}}_{j \geq k} \| \alpha_j \|_{\infty} < 1}$}.}
%
%\vspace{5pt}
%
%\noindent In this case, there exists $\tau > 0$ and  $K \geq 1$ such that ${\rm{sup}}_{j \geq K} \| \alpha_j \|_\infty \leq 1 - \tau$.  Hence, from $\bar{x}_k = \alpha_{k-1} \dt (\bar{x}_{k-1} - \bar{x}^\ast) + \bar{x}^\ast$, we have for $k \geq K$,
%\begin{eqnarray*}
%\| \bar{x}_k - \bar{x}^\ast \| \leq ({\rm{sup}}_{j \geq K-1} \| \alpha_j \|_\infty)^{k-K+1} \| \bar{x}_{K-1} - \bar{x}^\ast \| \leq (1 - \tau)^{k - K + 1} \| \bar{x}_{K-1} - \bar{x}^\ast \|.
%\end{eqnarray*}
%Hence, $\bar{x}_k \rightarrow \bar{x}^\ast$ as $k \rightarrow \infty$.

%\vspace{5pt}
%
%%\noindent {\bf{Case 2: $\mb{\lim_{k \rightarrow \infty}{\rm{sup}}_{j \geq k} \| \alpha_j \|_{\infty} = 1}$}.}
%%
%%\vspace{5pt}
%
%\noindent Let $\{\bar{x}_{j_k} \}$ be any convergent subsequence of $\{ \bar{x}_k \}$ that converges to $\bar{x}^{\ast \ast}$.  We now show that $\bar{x}^{\ast \ast} = \bar{x}^\ast$.  Therefore, the whole sequence $\{ \bar{x}_k \}$ converges to $\bar{x}^\ast$, and the theorem is proved. 

\vspace{5pt}

%\noindent We assume WLOG that $\{ \|\alpha_{j_k-1}\|_\infty \}$ converges with $\lim_{k \rightarrow \infty} \|\alpha_{j_k-1}\|_\infty \leq 1$.  
%If $\lim_{k \rightarrow \infty} \| \alpha_{j_k-1} \|_\infty < 1$, then from 
%\begin{eqnarray*}
%\| \bar{x}_{j_k} - \bar{x}^\ast \| \leq \Pi_{i = j_0}^{j_k-1} \|\alpha_{i} \|_\infty \| \bar{x}_{j_0} - \bar{x}^\ast \|,
%\end{eqnarray*}
%we have
%\begin{eqnarray*}
%\log \frac{ \| \bar{x}_{j_k} - \bar{x}^\ast \|}{ \| \bar{x}_{j_0} - \bar{x}^\ast \|} \leq \sum_{i=j_0}^{j_k - 1} \log \| \alpha_i \|_\infty,
%\end{eqnarray*}
%and the above contradicts $\lim_{k \rightarrow \infty} \| \alpha_{j_k-1} \|_\infty < 1$ if $\bar{x}_{j_k} \not\rightarrow \bar{x}^\ast$ as $k \rightarrow \infty$.  
%
%\vspace{5pt}
%
%\noindent Suppose $\lim_{k \rightarrow \infty} \| \alpha_{j_k-1} \|_\infty = 1$.  
\noindent WLOG, assume that $\{ \bar{x}_{j_k - 1} \}$ converges as well, and let the sequence converges to $\bar{x}^{\ast \ast \ast}$, for some $\bar{x}^{\ast \ast \ast} \in \Re^n$.  We have $\{ \bar{u}_{j_k - 1} \}$, where $\bar{u}_{j_k - 1} = J_{\bar{A}_0}(\bar{x}_{j_k - 1})$, also converges, and let $\bar{u}^{\ast \ast \ast}$ be its limit point.   By Lemma \ref{lem:alphakboundnD}, for $i = 1, \ldots, n$, we  either have $\lim_{k \rightarrow \infty} \bar{x}_{j_k-1}^i = \bar{x}^\ast_i$ or $\lim_{k \rightarrow \infty} | \alpha_k^i | = 1$.  Let $I := \{ i \ ; \ 1 \leq i \leq n, \lim_{k \rightarrow \infty} \alpha_{j_k-1}^i = -1\}$ and $I_1 := \{ i \ ; \ 1 \leq i \leq n, \lim_{k \rightarrow \infty} \alpha_{j_k-1}^i = 1\}$.  Observe that for $i \in I_1$, $\bar{x}^{\ast\ast}_i = \bar{x}^{\ast \ast \ast}_i$, and for $i \not\in I \cup I_1$, we have $\lim_{k \rightarrow \infty} \bar{x}_{j_k-1}^i = \bar{x}^\ast_i$.  Since $i \not\in I \cup I_1$, we must also have $\lim_{k \rightarrow \infty} \bar{x}^i_{j_k} = \bar{x}^\ast_i$, which we can show by contradiction.   Hence, for $i \not\in I \cup I_1$, $\bar{x}^{\ast \ast}_i = \bar{x}^{\ast \ast \ast}_i = \bar{x}^\ast_i$.

\vspace{5pt}

\noindent Consider the following expression, which has a similar expression in (\ref{ineq:barB0prime}):
\begin{eqnarray}
& & -\left(1 + \frac{2\beta ( \alpha_{j_k-1}^i - 1)}{2 + \beta + \epsilon} \right) \bar{u}_{j_k-1}^i (\bar{x}_{j_k-1}^i - \bar{x}^\ast_i) \nonumber \\
& & - \left( \frac{(1+ \beta)(\alpha_{j_k-1}^i - 1)}{2 + \beta + \epsilon} \right)\left( 1 + \frac{(1 + \beta)(\alpha_{j_k-1}^i - 1)}{2 + \beta + \epsilon}\right)(\bar{x}_{j_k-1}^i - \bar{x}^\ast_i)^2.  \label{exp:0}
\end{eqnarray}
%Note that a similar expression as the above expression appears in (\ref{ineq:barB0prime}).

%\vspace{5pt}

\noindent For $i \in I$, as $k \rightarrow \infty$, the expression (\ref{exp:0}) converges to
\begin{eqnarray}\label{exp:1}
\left(-1 + \frac{4 \beta}{2 + \beta + \epsilon} \right)\bar{u}^{\ast \ast \ast}_i(\bar{x}^{\ast \ast \ast}_i - \bar{x}^{\ast}_i) - \frac{2(1 + \beta)(\beta - \epsilon)}{(2 + \beta + \epsilon)^2}(\bar{x}^{\ast \ast \ast}_i - \bar{x}^\ast_i)^2.
\end{eqnarray}
\noindent For $i \in I_1$, as $k \rightarrow \infty$, it converges to
\begin{eqnarray}\label{exp:2}
\bar{u}^{\ast \ast \ast}_i(\bar{x}^\ast_i - \bar{x}^{\ast \ast \ast}_i).
\end{eqnarray}
Finally, for $i \not\in I \cup I_1$, it converges to $0$.

\vspace{5pt}

\noindent Hence, (\ref{ineq:barB0prime}) leads to the following:
\begin{eqnarray*}
\sum_{i \in I} \left[\left( -1 + \frac{4 \beta}{2 + \beta + \epsilon}\right) \bar{u}^{\ast \ast \ast}_i(\bar{x}^{\ast \ast \ast}_i  - \bar{x}^\ast_i) - \frac{2(1 + \beta)(\beta - \epsilon)}{(2 + \beta + \epsilon)^2}(\bar{x}^{\ast \ast \ast}_i - \bar{x}^\ast_i)^2 \right]  & & \\ 
+ \sum_{i \in I_1} \bar{u}^{\ast\ast\ast}_i(\bar{x}^\ast - \bar{x}^{\ast\ast\ast}) + \frac{1 - \beta}{1 + \beta} \| \bar{u}^{\ast \ast \ast} \|^2 & \geq & 0. 
\end{eqnarray*}
That is,
\begin{eqnarray}
\langle \bar{u}^{\ast\ast\ast}, \bar{x}^{\ast\ast\ast} - \bar{x}^\ast \rangle & \leq & \frac{1 - \beta}{1 + \beta} \| \bar{u}^{\ast \ast \ast} \|^2 + \sum_{i \in I} \left[ \frac{4 \beta}{2 + \beta + \epsilon} \bar{u}^{\ast \ast \ast}_i(\bar{x}^{\ast \ast \ast}_i  - \bar{x}^\ast_i) \right. \nonumber \\ 
& & \left. + \frac{2(1 + \beta)(\epsilon - \beta)}{(2 + \beta + \epsilon)^2}(\bar{x}^{\ast \ast \ast}_i - \bar{x}^\ast_i)^2 \right]. \label{ineq:combined}
\end{eqnarray}
Furthermore, by (\ref{ineq:barA0prime}), the inequality $\| \bar{u}_{j_k-1} \|^2 \leq \langle \bar{x}_{j_k-1} - \bar{x}^\ast, \bar{u}_{j_k-1} \rangle$ holds, and it tends to 
\begin{eqnarray}\label{ineq:combined1}
\| \bar{u}^{\ast \ast\ast} \|^2 \leq \langle \bar{u}^{\ast \ast \ast},  \bar{x}^{\ast \ast \ast} - \bar{x}^\ast \rangle,
\end{eqnarray}
as $k \rightarrow \infty$.  With (\ref{ineq:combined}), (\ref{ineq:combined1}), using Lemma \ref{lem:technical2}, where we let $\bar{x} = \bar{x}^\ast, \bar{\bar{x}} = \bar{x}^{\ast \ast \ast}, \bar{u} = \bar{u}^{\ast \ast \ast}$ in the lemma, we have $\bar{u}^{\ast \ast \ast} = 0$ and for $i \in I$, $\bar{x}^{\ast \ast\ast}_i = \bar{x}^{\ast}_i$.  Hence, for $i \in I$, we have $\bar{x}^{\ast \ast \ast}_i = \bar{x}^{\ast \ast}_i$. 

\vspace{5pt}

\noindent In summary, for $i \in I$, $\bar{x}^{\ast\ast\ast}_i = \bar{x}^{\ast \ast}_i = \bar{x}^\ast_i$, for $i \in I_1$, $\bar{x}^{\ast \ast \ast}_i = \bar{x}^{\ast \ast}_i$, while for $i \not\in I \cup I_1$, $\bar{x}^{\ast\ast\ast}_i = \bar{x}^{\ast\ast}_i = \bar{x}^\ast_i$.  Furthermore, $J_{\bar{A}_0}(\bar{x}^{\ast \ast \ast}) = \bar{u}^{\ast\ast\ast} = 0$.  By taking limit in (\ref{method:modifiedRelaxedPR2}), where $k$ in each subscript in (\ref{method:modifiedRelaxedPR2}) is replaced by $j_k$, we hence have $J_{\bar{B}_0}(-\bar{x}^{\ast \ast \ast}) = 0$.  Since $J_{\bar{A}_0}(\bar{x}^{\ast \ast \ast}) = J_{\bar{B}_0}(-\bar{x}^{\ast \ast \ast}) = 0$, (\ref{inc:barA0barB0}) is satisfied by $\bar{x}^{\ast \ast \ast}$, and therefore satisfied by $\bar{x}^{\ast\ast}$.
%, and by uniqueness of the solution to (\ref{inc:barA0barB0}), $\bar{x}^{\ast \ast \ast} = \bar{x}^\ast$, which further implies that $\bar{x}^{\ast \ast} = \bar{x}^\ast$.
%\vspace{5pt}
%
%\noindent We now consider the general case when ${\rm{dom}} \bar{A}_0$ can be unbounded.  In this case, define the maximal monotone operator $\bar{\bar{A}}_0$, with  ${\rm{dom}} \bar{\bar{A}}_0 $ bounded, to be such that
%\begin{eqnarray*}
%\bar{\bar{A}}_0(\bar{u}) = \left\{ \begin{array}{ll}
%									\bar{A}_0(\bar{u}) & ,{\rm{if}}\  \| \bar{u} \| \leq \| \bar{u}_0 \| + 1, \\
%									\emptyset & ,{\rm{otherwise}}.
%									\end{array} \right.
%\end{eqnarray*} 
%Given the starting iterate $\bar{x}_0$, let $\{\bar{x}_k \}$ be the sequence of iterates generated by (\ref{method:modifiedRelaxedPR2}) when $\bar{A}_0$ is itself.  It is easy to see that this sequence of iterates is the same as that generated by (\ref{method:modifiedRelaxedPR2}) when $\bar{A}_0 = \bar{\bar{A}}_0$ for the same starting iterate $\bar{x}_0$.   Therefore, by the earlier convergence arguments, $\{ \bar{x}_k \}$ converges.
\end{proof}

\vspace{10pt}

\noindent The following corollary follows immediately from the above theorem:
\begin{corollary}\label{cor:convergence}
Suppose $A$ or $B$ is single-valued, then for $\theta = 2 + \beta + \epsilon$, where $0 < \epsilon < \min \{ \beta, 1/\beta \}$, $\beta > 0$, the sequence $\{ x_k \}$ generated by the relaxed PR splitting method (\ref{RelaxedPR}) for a given initial iterate $x_0 \in \Re^n$ is a convergent sequence with limit point $x^\ast$, and $J_A(x^\ast)$ solves the monotone inclusion problem (\ref{MIP}).
\end{corollary}
\begin{proof}
We only need to prove that for $\theta = 2 + \beta + \epsilon$, where $0 < \epsilon < \min \{ \beta, 1/\beta \}$, $\beta > 0$, $\{ \bar{x}_k \}$ converges to $\bar{x}^\ast$.  Since, by Lemma \ref{lem:alphakboundnD}, we know that $\{ \bar{x}_k \}$ is bounded, this is achieved by showing that any convergent subsequence $\{\bar{x}_{j_k} \}$ of $\{ \bar{x}_k \}$ converges to $\bar{x}^\ast$.  By Theorem \ref{thm:maintheoremnD}, we see that the limit point of $\{ \bar{x}_{j_k}\}$ satisfies (\ref{inc:barA0barB0}).  Since $A$ or $B$ is single-valued, which implies that either $\bar{A}_0$ or $\bar{B}_0$ is single-valued, therefore, $\bar{x}^\ast$ is the only solution to (\ref{inc:barA0barB0}).  Hence, the limit point of $\{ \bar{x}_{j_k}\}$ is $\bar{x}^\ast$.
\end{proof}
{\textcolor{black}{\begin{remark}\label{convergencein1D}
We remark that when $n = 1$, for convergence of $\{ x_k \}$ given any initial iterate $x_0 \in \Re$, we do not require Assumption \ref{ass:Si} and only need $A, B$ to be maximal $\beta$-strongly monotone, when $\theta = 2 + \beta + \epsilon$, where $0 < \epsilon < \min \{\beta, 1/\beta\}, \beta > 0$ \cite{Sim}.  Hence, Assumption \ref{ass:Si} is only required when $n$ is greater than or equal to $2$ to show convergence of $\{ x_k \}$.  We believe that this assumption is always satisfied\footnote{\textcolor{black}{Preliminary numerical evidence  that assumption always holds is provided in Section \ref{sec:Numerical}.}}, although it is an open problem to show this for all practical problems.  Our approach to show convergence arises from the idea to solve a higher dimensional problem by ``reducing" it to a smaller dimensional problem ($n = 1$), and then applying our solution method when $n = 1$ \cite{Sim}, where we have convergence.  In the process of doing this, we realize that we need Assumption \ref{ass:Si} in order to show convergence in higher dimensions. In fact, there may exist other approaches that do not require this assumption to show convergence, and that convergence of iterates occurs over the considered range of $\theta$ for $A$ and $B$ maximal $\beta$-strongly monotone, and either of them single-valued. 
\end{remark}
}}

\section{Results on Convergence Rates}\label{sec:convergencerate}

%\noindent Recall that the relaxed PR splitting method is given by:
%\begin{eqnarray}\label{relaxedPR1}
%x_k = x_{k-1} + \theta (J_B(2J_A(x_{k-1}) - x_{k-1}) - J_A(x_{k-1})), \quad k \geq 1.
%\end{eqnarray}
%Let $u_{k-1} := J_A(x_{k-1})$ and $v_{k-1} := J_B(2u_{k-1} - x_{k-1})$.  Therefore, we have
%\begin{eqnarray}\label{inc:AB}
%(x_{k-1} - u_{k-1}, 2u_{k-1} - x_{k-1} - v_{k-1}) \in (A \times B)(u_{k-1},v_{k-1}).
%\end{eqnarray}
%Furthermore, we have $0 = J_A(x^\ast) = J_B(-x^\ast)$, which implies that
%\begin{eqnarray}\label{inc:ABast}
%(x^\ast, -x^\ast) \in (A \times B)(0, 0).
%\end{eqnarray}
%Hence, by strong monotonicity of $A \times B$ applied to (\ref{inc:AB}), (\ref{inc:ABast}), the following holds:
%\begin{eqnarray}\label{ineq:strongmonotonicity}
%\langle u_{k-1}, x_{k-1} - u_{k-1} - x^\ast \rangle + \langle v_{k-1}, 2u_{k-1} - x_{k-1} - v_{k-1} + x^\ast \rangle \geq \beta (\| u_{k-1} \|^2 + \| v_{k-1}\|^2).
%\end{eqnarray}

\noindent {\textcolor{black}{We start the section by relating $\bar{u}_k$ and $\bar{v}_k$ in the following way:
\begin{eqnarray}\label{rel:barukbarvk}
\bar{v}_{k} = - \gamma_{k} \bar{u}_{k} + \tau_{k} \bar{z}_{k},\ {\rm{where}}\ \| \bar{z}_{k} \| = 1 \ {\rm{and}}\ \langle \bar{u}_{k}, \bar{z}_{k} \rangle = 0.
\end{eqnarray}
Recall that $\bar{u}_k = J_{\bar{A}_0}(\bar{x}_k)$ and $\bar{v}_k = J_{\bar{B}_0}\left( \frac{2}{1+ \beta} \bar{u}_k - \bar{x}_k \right)$.  Note that there always exist $\gamma_{k}, \tau_{k}$ and $\bar{z}_{k}$ such that $\bar{v}_{k} = - \gamma_{k} \bar{u}_{k} + \tau_{k} \bar{z}_{k}$, $\| \bar{z}_{k} \| = 1$ and $\langle \bar{u}_{k}, \bar{z}_{k} \rangle = 0$ in (\ref{rel:barukbarvk}).  Also, $\gamma_k$ can be computed from $\bar{u}_k, \bar{v}_k$, and is equal to $- \langle \bar{u}_k, \bar{v}_k \rangle/\| \bar{u}_k \|^2$.  {\textcolor{black}{We have the following definition:
\begin{eqnarray}\label{def:bargammaunderlinegamma}
\bar{\gamma}:= \limsup_{k \rightarrow \infty} \gamma_k.
\end{eqnarray}
Furthermore, let $\bar{\gamma}_{\max}$ be the maximum of $\bar{\gamma}$ over all $\bar{x}_0 \in \Re^n$.  Note that $\bar{\gamma}_{\max}$ is only dependent on problem data and parameters.  It does not depend on the initial iterate and subsequent iterates generated.}}
}}

%stating the following technical assumption which is to hold throughout the section:
%\begin{assumption}\label{ass:iterationconvergencerate}
%Let $\bar{v}_{k} = - \gamma_{k} \bar{u}_{k} + \tau_{k} \bar{z}_{k}$, where $\| \bar{z}_{k} \| = 1$, and $\langle \bar{u}_{k}, \bar{z}_{k} \rangle = 0$.  As $k \rightarrow \infty$, we have $\gamma_{k} \rightarrow \gamma \in \Re$.
%\end{assumption}

%For $\theta = 2 + \beta + \epsilon$, where $0 \leq \epsilon < \min \{ \beta, 1/\beta \}$, $\beta > 0$, we have that $\bar{u}_{k}$ and $\bar{v}_{k}$ tends to zero, by results from the previous section and that in \cite{Monteiro}.  Assumption \ref{ass:iterationconvergencerate} means that the length of the projection of $\bar{v}_{k}$ along $\bar{u}_{k}$ is not approaching zero at the same rate as the length of $\bar{u}_{k}$, if the projection is pointing in the opposite direction to $\bar{u}_k$ as $k$ tends to infinity.

%\vspace{10pt}

\noindent To arrive at our convergence rates results, we observe from (\ref{ineq:barA0prime}), (\ref{ineq:barB0alternative}) that
\begin{eqnarray}\label{ineq:strongmonotonicity}
\langle \bar{u}_{k-1}, \bar{x}_{k-1} - \bar{u}_{k-1} - \bar{x}^\ast \rangle + \left\langle \bar{v}_{k-1}, \frac{2}{1+ \beta} \bar{u}_{k-1} - \bar{x}_{k-1} - \bar{v}_{k-1} + \bar{x}^\ast \right\rangle \geq 0.
\end{eqnarray} 

\noindent The following lemma enables us to arrive at Lemma \ref{lem:consequence}, which is key to proving our convergence rates results in Theorem \ref{thm:bargammaast}.
\begin{lemma}\label{lem:differenceukvk}
For $k \geq 1$, we have 
\begin{eqnarray}
(1 - \theta) \|\bar{u}_{k-1} - \bar{v}_{k-1} \|^2 & \leq & \frac{(1+\beta)^2}{\theta}( \| \bar{x}_{k-1} - \bar{x}^\ast \|^2 - \| \bar{x}_{k} - \bar{x}^\ast \|^2) \nonumber \\
& & + (1+\beta) \langle \bar{x}_{k-1} - \bar{x}^\ast, \bar{v}_{k-1} - \bar{u}_{k-1} \rangle - \beta( \| \bar{u}_{k-1} \|^2 + \| \bar{v}_{k-1} \|^2). \label{ineq:differenceukvk}
\end{eqnarray}
\end{lemma}
\begin{proof}
We have
\begin{eqnarray*}
& & \frac{(1+\beta)^2}{\theta}( \| \bar{x}_{k-1} - \bar{x}^\ast \|^2 - \| \bar{x}_{k} - \bar{x}^\ast \|^2) + (1+\beta) \langle \bar{x}_{k-1} - \bar{x}^\ast, \bar{v}_{k-1} - \bar{u}_{k-1} \rangle  \\
& & - \beta( \| \bar{u}_{k-1} \|^2 + \| \bar{v}_{k-1} \|^2) - (1 - \theta)  \|\bar{u}_{k-1} - \bar{v}_{k-1} \|^2 \\
& = & \frac{(1+\beta)^2}{\theta}\left( \| \bar{x}_{k-1} - \bar{x}^\ast \|^2 - \left\| (\bar{x}_{k-1} - \bar{x}^\ast) + \frac{\theta}{1 + \beta}(\bar{v}_{k-1} - \bar{u}_{k-1}) \right\|^2 \right)   \\
& & + (1+\beta) \langle \bar{x}_{k-1} - \bar{x}^\ast, \bar{v}_{k-1} - \bar{u}_{k-1} \rangle - \beta( \| \bar{u}_{k-1} \|^2 + \| \bar{v}_{k-1} \|^2) - (1 - \theta)  \|\bar{u}_{k-1} - \bar{v}_{k-1} \|^2 \\
& = & \frac{(1+\beta)^2}{\theta}\left( -\frac{2 \theta}{1 + \beta} \langle \bar{x}_{k-1} - \bar{x}^\ast, \bar{v}_{k-1} - \bar{u}_{k-1}) - \frac{\theta^2}{(1+ \beta)^2} \| \bar{u}_{k-1} - \bar{v}_{k-1} \|^2 \right)   \\
& & + (1+\beta) \langle \bar{x}_{k-1} - \bar{x}^\ast, \bar{v}_{k-1} - \bar{u}_{k-1} \rangle - \beta( \| \bar{u}_{k-1} \|^2 + \| \bar{v}_{k-1} \|^2) - (1 - \theta)  \|\bar{u}_{k-1} - \bar{v}_{k-1} \|^2 \\
& = & -(1+ \beta) \langle \bar{x}_{k-1} - \bar{x}^\ast, \bar{v}_{k-1} - \bar{u}_{k-1} \rangle - \beta( \| \bar{u}_{k-1} \|^2 + \| \bar{v}_{k-1} \|^2) - \| \bar{u}_{k-1} - \bar{v}_{k-1} \|^2 \\
%& = & (1 + \beta)\langle \bar{x}_{k-1} - \bar{x}^\ast, \bar{u}_{k-1} \rangle - (1+ \beta) \langle \bar{x}_{k-1} - \bar{x}^\ast, \bar{v}_{k-1} \rangle - \beta( \| \bar{u}_{k-1} \|^2 + \| \bar{v}_{k-1} \|^2) \\
%& & - \| \bar{u}_{k-1} - \bar{v}_{k-1} \|^2 \\
& = & (1 + \beta)\langle \bar{x}_{k-1} - \bar{x}^\ast, \bar{u}_{k-1} \rangle - (1+ \beta) \langle \bar{x}_{k-1} - \bar{x}^\ast, \bar{v}_{k-1} \rangle - (1 + \beta) (\| \bar{u}_{k-1} {\textcolor{black}{\|^2}} + \| \bar{v}_{k-1} \|^2) \\
& &  + 2 \langle \bar{u}_{k-1}, \bar{v}_{k-1} \rangle \\
& \geq & 0,
\end{eqnarray*}
where the first equality follows from Proposition \ref{prop:simpleobservation}, and the inequality follows from (\ref{ineq:strongmonotonicity}).
\end{proof}

\vspace{10pt}

\noindent The following lemma follows from the above lemma.
\begin{lemma}\label{lem:consequence}
For $k \geq 1$, we have
\begin{eqnarray} \label{ineq:consequence2}
4 \beta \langle \bar{u}_{k-1}, \bar{v}_{k-1} \rangle + (\beta - \epsilon)\| \bar{u}_{k-1} - \bar{v}_{k-1} \|^2 \leq  \frac{(1+\beta)^2}{\theta}( \| \bar{x}_{k-1} - \bar{x}^\ast \|^2 - \| \bar{x}_k - \bar{x}^\ast \|^2) ,
\end{eqnarray}
where $\theta = 2 + \beta + \epsilon$, $0 \leq \epsilon < \min \{\beta, 1/\beta \}$, $\beta > 0$.
\end{lemma}
\begin{proof}
By expanding the left-hand side of (\ref{ineq:differenceukvk}) in Lemma \ref{lem:differenceukvk}, using $\theta = 2 + \beta + \epsilon$, and upon algebraic manipulations, we obtain the following:
\begin{eqnarray}
-(1 + \epsilon)( \| \bar{u}_{k-1} \|^2 + \| \bar{v}_{k-1} \|^2) & \leq & \frac{(1 + \beta)^2}{\theta} ( \| \bar{x}_{k-1} - \bar{x}^\ast \|^2 - \| \bar{x}_k - \bar{x}^\ast \|^2) \nonumber \\
& & + (1 + \beta) \langle \bar{x}_{k-1} - \bar{x}^\ast, \bar{v}_{k-1} - \bar{u}_{k-1} \rangle - 2(1 + \beta + \epsilon) \langle \bar{u}_{k-1}, \bar{v}_{k-1} \rangle.  \nonumber \\
& & \label{ineq:consequence3}
\end{eqnarray}
Observe that
\begin{eqnarray*}
& & (1 + \beta) \langle \bar{x}_{k-1} - \bar{x}^\ast, \bar{v}_{k-1} - \bar{u}_{k-1} \rangle - 2(1 + \beta + \epsilon) \langle \bar{u}_{k-1}, \bar{v}_{k-1} \rangle \\
& \leq & - 2(\beta + \epsilon) \langle \bar{u}_{k-1}\, \bar{v}_{k-1} \rangle - (1+ \beta)(\| \bar{u}_{k-1} \|^2 + \| \bar{v}_{k-1} \|^2),
%& = & (\beta + \epsilon) \| \bar{u}_{k-1} - \bar{v}_{k-1} \|^2 - (1 + 2\beta + \epsilon) (\| \bar{u}_{k-1} \|^2 + \| \bar{v}_{k-1} \|^2),
\end{eqnarray*}
where the inequality follows from (\ref{ineq:barA0prime}) and (\ref{ineq:barB0alternative}).  Substituting the above inequality into (\ref{ineq:consequence3}) and upon algebraic manipulations, we have
\begin{eqnarray*}
2(\beta + \epsilon) \langle \bar{u}_{k-1}\, \bar{v}_{k-1} \rangle + (\beta - \epsilon) (\| \bar{u}_{k-1} \|^2 + \| \bar{v}_{k-1} \|^2) \leq \frac{(1 + \beta)^2}{\theta} ( \| \bar{x}_{k-1} - \bar{x}^\ast \|^2 - \| \bar{x}_k  - \bar{x}^\ast \|^2). 
\end{eqnarray*}
The lemma then follows from the above by observing that $\| \bar{u}_{k-1} \|^2 + \| \bar{v}_{k-1} \|^2 = \| \bar{u}_{k-1} - \bar{v}_{k-1} \|^2 + 2 \langle \bar{u}_{k-1}, \bar{v}_{k-1} \rangle$.
\end{proof}

\vspace{10pt}

% particular, we have as $k \rightarrow \infty$,
%\begin{eqnarray*}
%\frac{\| \bar{v}_{k-1} \|^2}{\| \bar{u}_{k-1} \|^2} = (\gamma_{k-1})^2 + \frac{\tau_{k-1}^2}{\| \bar{u}_{k-1} \|^2} \rightarrow \gamma^2, \quad \frac{\langle \bar{v}_{k-1}, \bar{w}_{k-1} \rangle}{\| \bar{u}_{k-1} \|} = \frac{\tau_{k-1}}{\| \bar{u}_{k-1} \|} \rightarrow 0.
%\end{eqnarray*}
%Hence, Assumption \ref{ass:iterationconvergencerate} implies that as $\bar{u}_{k-1}$ and $\bar{v}_{k-1}$ approach zero, they are doing so in such a way that their lengths are not the same, and they are also getting more ``aligned" to each other.  The boundedness of $\frac{ \| \bar{x}_{k-1} - \bar{x}^\ast \|}{\| \bar{u}_{k-1} \|}$ can  be achieved by for example having $A$ or $B$ to be Lipschitz continuous - we provide the proof of this in the appendix.  Certainly, having $A$ or $B$ Lipschitz continuous implies Assumption \ref{ass:single-valued}, and is hence a stronger requirement on $A$, $B$ than that in Assumption \ref{ass:single-valued}.  

%\vspace{10pt}

\noindent We need the following lemma to prove results on the {\textcolor{black}{$R$}}-linear convergence rate of $\| x_{k} - x^\ast \|$ in Theorem \ref{thm:bargammaast}:

\begin{lemma}\label{lem:Lipschitz}
Suppose  $A$ or $B$ is Lipschitz continuous with Lipschitz constant $L (\geq 0)$, then for all $k \geq 0$,
\begin{eqnarray}\label{ineq:Lipschitz}
\| \bar{x}_k - \bar{x}^\ast \| \leq \bar{L} (\| \bar{u}_k \| + \| \bar{v}_k \|),
\end{eqnarray}
where $\bar{L} := \max\left\{\frac{L}{1 + \beta} + 1, \frac{2}{1+\beta} \right\}$.
\end{lemma}
\begin{proof}
\noindent Since $A$ or $B$ is Lipschitz continuous with Lipschitz continuous $L (\geq 0)$, it is clear that either $\bar{A}_0$ or $\bar{B}_0$ is also Lipschitz continuous with Lipschitz constant $\frac{L}{1 + \beta}$.

\vspace{5pt}

\noindent Suppose $\bar{A}_0$ is Lipschitz continuous.  We have from (\ref{inc:barA0barB0}) and (\ref{inc:barA0prime}) that
\begin{eqnarray*}
\| \bar{x}_k - \bar{u}_k - \bar{x}^\ast \| \leq \frac{L}{1 + \beta} \| \bar{u}_k \| .
\end{eqnarray*} 
It then follows that (\ref{ineq:Lipschitz}) holds for all $k \geq 0$ from the above by the triangle inequality.

\vspace{5pt}

\noindent Suppose $\bar{B}_0$ is Lipschitz continuous.  Then, from (\ref{inc:barA0barB0}) and (\ref{inc:barB0prime}) , using (\ref{def:barwk0}), the following holds:
\begin{eqnarray}\label{ineq:Lipschitz3}
\left\| \frac{2}{1+ \beta}\bar{u}_k - \bar{v}_k - (\bar{x}_k - \bar{x}^\ast) \right\| \leq \frac{L}{1 + \beta} \| \bar{v}_k \|. 
\end{eqnarray}
The result then follows by the triangle inequality.
\end{proof}

\begin{theorem}\label{thm:bargammaast}
Let $\{ x_k \}$ be generated by the relaxed PR splitting method (\ref{RelaxedPR}) given an initial iterate $x_0 \in \Re^n$, where $\theta = 2 + \beta + \epsilon$, $0 \leq \epsilon < \min \{\beta, 1/\beta \}$, $\beta > 0$.  {\textcolor{black}{For 
\begin{eqnarray}\label{ineq:gammabound1}
\bar{\gamma}_{\max} < \frac{\sqrt{\beta} - \sqrt{\epsilon}}{\sqrt{\beta} + \sqrt{\epsilon}}
\end{eqnarray}}}
there exists $i \leq k$ and {\textcolor{black}{$i \geq k_0$}}, where $k \geq {\textcolor{black}{2}}k_0$, and 
\begin{eqnarray}\label{def:M}
M := \min \left\{ M_1 (\beta - \epsilon), {\textcolor{black}{\frac{f(\bar{\gamma}_{\max})(1 - M_1)(\beta - \epsilon)}{8( \beta + \epsilon)\max\{\bar{\gamma},1\}}}} \right\} > 0,
\end{eqnarray}
{\textcolor{black}{where $f(\bar{\gamma}_{\max}) = - 2( \beta + \epsilon) \bar{\gamma}_{\max} + (\beta - \epsilon)( 1 + \bar{\gamma}_{\max}^2) > 0$ }}and $0 < M_1 < 1$, such that
\begin{eqnarray*}
\| x_i - x_{i-1} \| \leq \left(\frac{{\textcolor{black}{2}}\sqrt{\theta}\| x_{k_0-1} - x^\ast \|}{\sqrt{M}}\right) \frac{1}{\sqrt{k}}.
\end{eqnarray*}
Here, $k_0 \geq 1$ depends only on problem data and parameters.  Furthermore, if $A$ or $B$ is Lipschitz continuous with Lipschitz constant $L (\geq 0)$, then for all $k \geq k_0$,
\begin{eqnarray*}
\| x_k - x^\ast \| \leq \left[\sqrt{1 - \frac{M \theta}{2(1 + \beta)^2\bar{L}^2}}\right] \| x_{k-1} - x^\ast \|,
\end{eqnarray*}
where $\bar{L}$ is defined in Lemma \ref{lem:Lipschitz}.
\end{theorem}
\begin{proof}
Given that $\bar{v}_{k-1} = -\gamma_{k-1} \bar{u}_{k-1} + \tau_{k-1} \bar{z}_{k-1}$, where $\| \bar{z}_{k-1} \| = 1$, $\langle \bar{u}_{k-1}, \bar{z}_{k-1} \rangle = 0$, we have
\begin{eqnarray}
& & \frac{4 \beta \langle \bar{u}_{k-1}, \bar{v}_{k-1} \rangle}{\| \bar{u}_{k-1} \|^2} + \frac{(\beta - \epsilon)\| \bar{u}_{k-1} - \bar{v}_{k-1} \|^2}{ \| \bar{u}_{k-1} \|^2} \nonumber \\
 & = & \frac{2(\beta + \epsilon) \langle \bar{u}_{k-1}, \bar{v}_{k-1} \rangle}{\| \bar{u}_{k-1} \|^2} + \frac{(\beta - \epsilon)( \| \bar{u}_{k-1} \|^2 + \| \bar{v}_{k-1} \|^2)}{ \| \bar{u}_{k-1} \|^2} \nonumber \\
 & = & -2( \beta + \epsilon) \gamma_{k-1} + (\beta - \epsilon)\left[1 + \gamma_{k-1}^2 + \frac{\tau_{k-1}^2}{ \| \bar{u}_{k-1} \|^2}\right].  \label{eq:iterationconvergencerate}
\end{eqnarray}
Let $M_1 \in (0,1)$ be given.  Consider
\begin{eqnarray*}
I := \left\{ k \geq 1 \ ;\  -2( \beta + \epsilon) \gamma_{k-1} + (1 - M_1)(\beta - \epsilon)\left[1 + \gamma_{k-1}^2 + \frac{\tau_{k-1}^2}{ \| \bar{u}_{k-1} \|^2}\right] \geq  0 \right\}.
\end{eqnarray*} 
Therefore, from (\ref{eq:iterationconvergencerate}), we see that for $k \in I$, we have
\begin{eqnarray}
M_1 (\beta - \epsilon)( \| \bar{u}_{k-1}\|^2 + \| \bar{v}_{k-1} \|^2) & = & M_1( \beta - \epsilon)((1 + \gamma_{k-1}^2)\| \bar{u}_{k-1} \|^2 + \tau_{k-1}^2) \nonumber \\
& \leq & 4\beta \langle \bar{u}_{k-1}, \bar{v}_{k-1} \rangle + (\beta - \epsilon) \| \bar{u}_{k-1} - \bar{v}_{k-1} \|^2.  \label{ineq:iterationconvergencerate}
\end{eqnarray}
%and 
%\begin{eqnarray}\label{ineq:iterationconvergencerateprime}
%M_1 (\beta - \epsilon) (1 + \gamma_{k-1})^2\| \bar{u}_{k-1} \|^2 \leq 4\beta \langle \bar{u}_{k-1}, \bar{v}_{k-1} \rangle + (\beta - \epsilon) \| \bar{u}_{k-1} - \bar{v}_{k-1} \|^2.
%\end{eqnarray}
Let us now consider $k \not\in I$.  From (\ref{eq:iterationconvergencerate}), we have
\begin{eqnarray}\label{ineq:iterationconvergencerate2}
\frac{4 \beta \langle \bar{u}_{k-1}, \bar{v}_{k-1} \rangle}{\| \bar{u}_{k-1} \|^2} + \frac{(\beta - \epsilon)\| \bar{u}_{k-1} - \bar{v}_{k-1} \|^2}{ \| \bar{u}_{k-1} \|^2} \geq  -2(\beta + \epsilon) \gamma_{k-1} + (\beta - \epsilon)(1 + \gamma_{k-1}^2). 
\end{eqnarray} 
Let
\begin{eqnarray*}
f(\gamma) := - 2( \beta + \epsilon) \gamma + (\beta - \epsilon)( 1 + \gamma^2).
\end{eqnarray*}
{\textcolor{black}{We have $f(\bar{\gamma}_{\max}) > 0$ if $\bar{\gamma}_{\max}$ satisfies (\ref{ineq:gammabound1}).}}
We see that since {\textcolor{black}{$\limsup_{k \rightarrow \infty} \gamma_k = \bar{\gamma} \leq \bar{\gamma}_{\max}$}}, for $k$ sufficiently large, say $k \geq k_0$, we have $-2( \beta + \epsilon) \gamma_{k-1} + (\beta - \epsilon)(1 + \gamma_{k-1}^2) \geq {\textcolor{black}{f(\bar{\gamma}_{\max})/2 > 0}}$, and hence from (\ref{ineq:iterationconvergencerate2}), we have
\begin{eqnarray}\label{ineq:iterationconvergencerate3}
4 \beta \langle \bar{u}_{k-1}, \bar{v}_{k-1} \rangle + (\beta - \epsilon) \| \bar{u}_{k-1} - \bar{v}_{k-1} \|^2 \geq {\textcolor{black}{\frac{f(\bar{\gamma}_{\max})}{2}}} \| \bar{u}_{k-1} \|^2.
\end{eqnarray}
%Now,
%\begin{eqnarray*}
%\frac{\| \bar{u}_{k-1}\|^2  + \|\bar{v}_{k-1} \|^2}{ \| \bar{u}_{k-1} \|^2} = 1 + \gamma_{k-1}^2 + \frac{\tau_{k-1}^2}{\| \bar{u}_{k-1} \|^2}. 
%\end{eqnarray*}
For $k \not\in I$ {\textcolor{black}{and $k \geq k_0$}}, by definition of $I$, we have that 
\begin{eqnarray*}
\frac{\| \bar{u}_{k-1} \|^2 + \| \bar{v}_{k-1} \|^2}{ \| \bar{u}_{k-1} \|^2} = 1 + \gamma_{k-1}^2 + \frac{\tau_{k-1}^2}{\| \bar{u}_{k-1} \|^2} \leq \frac{2(\beta + \epsilon)\gamma_{k-1}}{(1 - M_1)(\beta - \epsilon)} \leq \frac{{\textcolor{black}{4}} ( \beta + \epsilon) {\textcolor{black}{\max \{\bar{\gamma},1 \}}}}{(1 - M_1)(\beta - \epsilon)},
\end{eqnarray*}
Therefore, from (\ref{ineq:iterationconvergencerate3}) which holds when $k \geq k_0$, for $k \not\in I$, we have {\textcolor{black}{
\begin{eqnarray}
& & \frac{f(\bar{\gamma}_{\max})(1 - M_1)(\beta - \epsilon)}{8 ( \beta + \epsilon) \max \{\bar{\gamma},1\}} (\| \bar{u}_{k-1}\|^2 + \|\bar{v}_{k-1} \|^2)  \nonumber \\
& \leq & \frac{f(\bar{\gamma}_{\max})}{2} \| \bar{u}_{k-1} \|^2  \leq   4\beta \langle \bar{u}_{k-1}, \bar{v}_{k-1} \rangle + (\beta - \epsilon) \| \bar{u}_{k-1} - \bar{v}_{k-1} \|^2. \label{ineq:iterationconvergencerate4}
\end{eqnarray}
}}Combining (\ref{ineq:iterationconvergencerate}) and (\ref{ineq:iterationconvergencerate4}), where the latter holds when $k \geq k_0$, then 
%\begin{eqnarray*}
%M := \min \left\{ M_1 (\beta - \epsilon), \frac{f(\gamma)}{2 \max \left\{ \frac{4 ( \beta + \epsilon) \gamma}{(1 - M_1)(\beta - \epsilon)}, 1 \right\}} \right\} > 0
%\end{eqnarray*}
%is such that
\begin{eqnarray}\label{ineq:iterationconvergencerate5}
M (\| \bar{u}_{k-1}\|^2 + \| \bar{v}_{k-1} \|^2) \leq 4\beta \langle \bar{u}_{k-1}, \bar{v}_{k-1} \rangle + (\beta - \epsilon) \| \bar{u}_{k-1} - \bar{v}_{k-1} \|^2,
\end{eqnarray}
where $M$ is defined by (\ref{def:M}).  Hence, from (\ref{ineq:consequence2}) in Lemma \ref{lem:consequence}, we have, for $k \geq k_0$,
\begin{eqnarray}\label{ineq:Rlinear2}
M (\| \bar{u}_{k-1} \|^2 + \| \bar{v}_{k-1} \|^2) \leq \frac{(1+ \beta)^2}{\theta} (\| \bar{x}_{k-1} - \bar{x}^\ast \|^2 - \| \bar{x}_k - \bar{x}^\ast\|^2).
\end{eqnarray}
Therefore, using $\| \bar{u}_{k-1} - \bar{v}_{k-1} \|^2 \leq 2( \| \bar{u}_{k-1} \|^2 + \| \bar{v}_{k-1} \|^2)$, Proposition \ref{prop:simpleobservation} and the definitions of $\bar{x}_k, \bar{x}^\ast$, we have from (\ref{ineq:Rlinear2})
\begin{eqnarray*}
{\textcolor{black}{(k - k_0)}} \inf_{{\textcolor{black}{k_0}} \leq i \leq k} \| x_i - x_{i-1} \|^2 \leq \theta^2 \sum_{i = k_0}^{k}\| \bar{u}_{i-1} - \bar{v}_{i-1} \|^2 & \leq & \frac{2\theta (1+\beta)^2}{M} ( \| \bar{x}_{k_0 - 1} - \bar{x}^\ast \|^2) \\
& = & \frac{2\theta \| x_{k_0-1} - x^\ast \|^2}{M}.
\end{eqnarray*}
{\textcolor{black}{Hence, for $k \geq 2k_0$, 
\begin{eqnarray*}
\inf_{k_0 \leq i \leq k} \| x_i - x_{i-1} \|^2 \leq \frac{2 \theta \| x_{k_0-1} - x^\ast \|^2}{M (k - k_0)} \leq \frac{4 \theta \| x_{k_0-1} - x^\ast \|^2}{Mk}.
\end{eqnarray*}
}}The first result in the theorem then follows from the above.

\vspace{5pt}

\noindent Suppose $A$ or $B$ is Lipschitz continuous with Lipschitz constant $L$, then  when  $k \geq k_0$, we have 
\begin{eqnarray*}
\frac{M}{2 \bar{L}^2} \| \bar{x}_{k-1} - \bar{x}^\ast \|^2 & \leq & \frac{M}{2} (\|\bar{u}_{k-1} \| + \| \bar{v}_{k-1} \|)^2  \leq   M (\| \bar{u}_{k-1} \|^2 + \|\bar{v}_{k-1}\|^2) \\
& \leq & 4 \beta \langle \bar{u}_{k-1}, \bar{v}_{k-1} \rangle + (\beta - \epsilon) \| \bar{u}_{k-1} - \bar{v}_{k-1} \|^2 \\
& \leq & \frac{(1+ \beta)^2}{\theta} (\| \bar{x}_{k-1} - \bar{x}^\ast \|^2 - \| \bar{x}_k - \bar{x}^\ast\|^2),
\end{eqnarray*}
where the first inequality follows from (\ref{ineq:Lipschitz}) in Lemma \ref{lem:Lipschitz}, the third inequality follows from (\ref{ineq:iterationconvergencerate5}), which holds when $k \geq k_0$, and the last inequality follows from (\ref{ineq:consequence2}) in Lemma \ref{lem:consequence}.
The second result in the theorem then follows from the definitions of $\bar{x}_k, \bar{x}^\ast$,  and the above.
\end{proof}

{\textcolor{black}{
\begin{corollary}\label{cor:ALipschitz}
{\textcolor{black}{Let $\{ x_k \}$ be generated by the relaxed PR splitting method (\ref{RelaxedPR}) given an initial iterate $x_0 \in \Re^n$, where $\theta = 2 + \beta + \epsilon$, $0 \leq \epsilon < \min \{\beta, 1/\beta \}$, $\beta > 0$. Suppose $A$ is Lipschitz continuous with Lipschitz constant $L (> 0)$ such that
\begin{eqnarray}\label{ineq:L}
L < 2 \sqrt{2(1 + \beta)}\sqrt{\frac{ \sqrt{\beta} - \sqrt{\epsilon}}{\sqrt{\beta} + \sqrt{\epsilon}}} - (1 + \beta).
\end{eqnarray}
In particular, when we choose $\epsilon$ to be zero (that is, $\theta = 2 + \beta$), we require
\begin{eqnarray*}
L < 2 \sqrt{2(1+\beta)} - (1 + \beta),
\end{eqnarray*}
where we need $\beta < 7$ for $2 \sqrt{2(1+\beta)} - (1 + \beta)$ to be positive.
Then there exists $i \leq k$ and {\textcolor{black}{$i \geq k_0$}}, where $k \geq {\textcolor{black}{2}}k_0$, and $M$ given by (\ref{def:M}), such that
\begin{eqnarray*}
\| x_i - x_{i-1} \| \leq \left(\frac{{\textcolor{black}{2}}\sqrt{\theta}\| x_{k_0-1} - x^\ast \|}{\sqrt{M}}\right) \frac{1}{\sqrt{k}},
\end{eqnarray*}
and for all $k \geq k_0$,
\begin{eqnarray*}
\| x_k - x^\ast \| \leq \left[\sqrt{1 - \frac{M \theta}{2(1 + \beta)^2\bar{L}^2}}\right] \| x_{k-1} - x^\ast \|,
\end{eqnarray*}
where $\bar{L}$ is defined in Lemma \ref{lem:Lipschitz}.  Here, $k_0 \geq 1$ depends only on problem data and parameters.}}
\end{corollary}
\begin{proof}
{\textcolor{black}{Given that $A$ is Lipschitz continuous with Lipschitz constant $L$ that satisfies (\ref{ineq:L}).  The corollary is proved by showing that $\bar{\gamma}_{\max}$ satisfies (\ref{ineq:gammabound1}).  Since $A$ is Lipschitz continuous with Lipschitz continuous $L$, then $\bar{A}_0$ is also Lipschitz continuous with Lipschitz constant $\frac{L}{1 + \beta}$.  From (\ref{inc:barA0barB0}) and (\ref{inc:barA0prime}), we then have
\begin{eqnarray}\label{ineq:Alipschitz}
\| \bar{x}_k - \bar{u}_k - \bar{x}^\ast \| \leq \frac{L}{1 + \beta} \| \bar{u}_k \|.
\end{eqnarray} 
Hence,
\begin{eqnarray*}
\| \bar{x}_k - \bar{x}^\ast\|^2 - 2 \langle \bar{x}_k - \bar{x}^\ast, \bar{u}_k \rangle + \| \bar{u}_k \|^2 \leq \frac{L^2}{(1 + \beta)^2} \| \bar{u}_k \|^2.
\end{eqnarray*}
Therefore,
\begin{eqnarray}\label{ineq:constraint}
\left( \frac{\| \bar{x}_k - \bar{x}^\ast\|}{\| \bar{u}_k \|} \right)^2 \leq \left(\frac{L^2}{(1 + \beta)^2} - 1 \right) + 2 \frac{\| \bar{x}_k - \bar{x}^\ast \|}{\| \bar{u}_k \|}.
\end{eqnarray}
%Hence, for all $k \geq 0$,
%\begin{eqnarray*}
%\gamma_k = -\frac{\langle \bar{u}_k, \bar{v}_k \rangle}{\| \bar{u}_k \|^2} \leq \frac{\|\bar{v}_k\|}{\| \bar{u}_k \|} \leq \frac{ \left\| \frac{2}{1 + \beta}\bar{u}_k - \bar{x}_k + \bar{x}^\ast \right\|}{ \| \bar{u}_k \|} \leq \frac{1 - \beta + L}{1 + \beta},
%\end{eqnarray*}
%where the second inequality follows from $\bar{v}_k = J_{\bar{B}_0}\left( \frac{2}{1+ \beta} \bar{u}_k - \bar{x}_k \right)$, $0 = J_{\bar{B}_0}(-\bar{x}^\ast)$ and the nonexpansiveness of $J_{\bar{B}_0}(\cdot)$.  Therefore, by definition of $\bar{\gamma}_{\max}$, which is the maximum over all $x_0 \in \Re^n$ of $\bar{\gamma} = \limsup_{k \rightarrow \infty} \gamma_k$, we have
%\begin{eqnarray*}
%\bar{\gamma}_{\max} \leq  \frac{1 - \beta + L}{1 + \beta},
%\end{eqnarray*}
%and (\ref{ineq:gammabound1}) is satisfied if
%\begin{eqnarray*}
%\frac{1 - \beta + L}{1 + \beta} < \frac{\sqrt{\beta} - \sqrt{\epsilon}}{\sqrt{\beta} + \sqrt{\epsilon}}.
%\end{eqnarray*}
%The result then follows by algebraic manipulations.
From (\ref{ineq:barB0alternative}), we have
\begin{eqnarray}\label{ineq:bargammak_LipschitzA}
-\frac{2 \langle \bar{u}_k, \bar{v}_k \rangle}{(1+\beta)\| \bar{u}_k \|^2} \leq  -\left( \frac{ \| \bar{v}_k \|}{\| \bar{u}_k \|} \right)^2 + \left( \frac{ \| \bar{v}_k \|}{\| \bar{u}_k \|} \right) \left( \frac{ \| \bar{x}_k - \bar{x}^\ast \|}{\| \bar{u}_k \|} \right).
\end{eqnarray}
With (\ref{ineq:constraint}), by Proposition \ref{prop:appendix0} applied to the expression on right-hand side of the above inequality, where $x = \frac{\| \bar{v}_k \|}{\| \bar{u}_k \|}$ and $y = \frac{\| \bar{x}_k - \bar{x}^\ast \|}{ \| \bar{u}_k \|}$ in the propostion, we have from (\ref{ineq:bargammak_LipschitzA}) that for all $k \geq 0$,
\begin{eqnarray*}
\gamma_k = - \frac{\langle \bar{u}_k, \bar{v}_k \rangle}{\| \bar{u}_k \|^2} \leq \frac{1 + \beta}{8}\left( 1 + \frac{L}{1 + \beta} \right)^2.
\end{eqnarray*}
Therefore, 
\begin{eqnarray*}
\bar{\gamma}_{\max} \leq \frac{1 + \beta}{8}\left( 1 + \frac{L}{1 + \beta} \right)^2.
\end{eqnarray*}
Hence, since $L$ satisfies (\ref{ineq:L}), $\bar{\gamma}_{\max}$ satisfies (\ref{ineq:gammabound1}) in Theorem \ref{thm:bargammaast}, and the corollary then follows from the theorem.
}}
\end{proof}
}}

\begin{corollary}\label{cor:BLipschitz}
{\textcolor{black}{Let $\{ x_k \}$ be generated by the relaxed PR splitting method (\ref{RelaxedPR}) given an initial iterate $x_0 \in \Re^n$, where $\theta = 2 + \beta + \epsilon$, $0 \leq \epsilon < \min \{\beta, 1/\beta \}$, $\beta > 0$. Suppose $B$ is Lipschitz continuous with Lipschitz constant $L (> 0)$ such that
\begin{eqnarray*}
L < 1+\beta.
\end{eqnarray*}
Then there exists $i \leq k$ and {\textcolor{black}{$i \geq k_0$}}, where $k \geq {\textcolor{black}{2}}k_0$, and $M$ given by (\ref{def:M}), such that
\begin{eqnarray*}
\| x_i - x_{i-1} \| \leq \left(\frac{{\textcolor{black}{2}}\sqrt{\theta}\| x_{k_0-1} - x^\ast \|}{\sqrt{M}}\right) \frac{1}{\sqrt{k}},
\end{eqnarray*}
and for all $k \geq k_0$,
\begin{eqnarray*}
\| x_k - x^\ast \| \leq \left[\sqrt{1 - \frac{M \theta}{2(1 + \beta)^2\bar{L}^2}}\right] \| x_{k-1} - x^\ast \|,
\end{eqnarray*}
where $\bar{L}$ is defined in Lemma \ref{lem:Lipschitz}.  Here, $k_0 \geq 1$ depends only on problem data and parameters.}}
\end{corollary}
\begin{proof}
{\textcolor{black}{Since $B$ is Lipschitz continuous with Lipschitz constant $L (> 0)$, we have $\bar{B}_0$ is Lipschitz continuous with Lipschitz constant $\frac{L}{1 + \beta}$, and the following holds:
\begin{eqnarray*}
\left\| \frac{2}{1+ \beta}\bar{u}_k - \bar{v}_k - (\bar{x}_k - \bar{x}^\ast) \right\| \leq \frac{L}{1 + \beta} \| \bar{v}_k \|. 
\end{eqnarray*}
Squaring both sides of the above inequality, upon algebraic manipulations and applying Cauchy-Schwartz inequality, we obtain:
\begin{eqnarray*}
- \frac{4\langle \bar{u}_k, \bar{v}_k \rangle}{(1 + \beta) \| \bar{u}_k \|^2}  & \leq & - \left(\frac{\| \bar{x}_k - \bar{x}^\ast \|}{\| \bar{u}_k \|} \right)^2  +  \frac{4}{1+\beta} \left(\frac{\| \bar{x}_k - \bar{x}^\ast \|}{\| \bar{u}_k \|}\right) + 2 \left( \frac{\| \bar{v}_k \|}{\| \bar{u}_k \|} \right) \left( \frac{ \| \bar{x}_k - \bar{x}^\ast \|}{ \| \bar{u}_k \|} \right) \\
& & + \left[\frac{L^2}{(1+\beta)^2} - 1 \right] \left(\frac{\|\bar{v}_k \|}{\| \bar{u}_k \|} \right)^2 - \frac{4}{(1 + \beta)^2} .
\end{eqnarray*}
By Proposition \ref{prop:appendix}, when $L < 1+ \beta$, the expression on the right-hand side of the above inequality is less than or equal to zero.  This implies that for all $k \geq 0$, $\gamma_k \leq 0$, which further implies that $\bar{\gamma}_{\max} \leq 0$.  Condition (\ref{ineq:gammabound1}) in Theorem \ref{thm:bargammaast} is hence satisfied, and the corollary follows.}}
\end{proof}

\vspace{10pt}

\noindent {\textcolor{black}{Note that in the above corollaries, we do not attempt to optimize the bounds on the Lipschitz constants for $A$ and $B$ for the results to hold, and finding the optimized bounds for these is left to interested readers.}}

\begin{remark}\label{rem:pointwiseconvergence}
Note that for $\theta = 2 + \beta, \beta >0$, by Lemma 4.5 in \cite{Monteiro}, $\{ \| x_{k}- x_{k-1} \|\}$ is nonincreasing.  Hence, the above theorem and corollaries imply that  when $\theta = 2 + \beta$, {\textcolor{black}{for $\bar{\gamma}_{\max} < 1$ or for $A$ Lipschitz continuous with Lipschitz constant $L < 2$ or for $B$ Lipschitz continuous with Lipschitz constant $L < 1 + \beta$}}, the pointwise convergence rate of $\| x_k - x_{k-1} \|$ for iterates $\{x_k\}$ generated using (\ref{RelaxedPR}) on (\ref{MIP}) is $\mathcal{O}(1/\sqrt{k})$. This is a partial extension of the pointwise convergence rate result in \cite{Monteiro} from $\theta \in (0, 2 + \beta)$ to $\theta = 2 + \beta$.  It is still an open problem whether a full extension is possible.
\end{remark}

\begin{remark}\label{rem:usualcase}
{\textcolor{black}{Observe from the proof of Theorem \ref{thm:bargammaast} that when $\theta \in (0, 2 + \beta)$ (that is, $\epsilon < 0$), we have pointwise convergence rate for $\| x_k - x_{k-1} \|$ of $\mathcal{O}(1/\sqrt{k})$ and $R$-linear convergence for $\|x_k - x^\ast \|$ without the need for (\ref{ineq:gammabound1}) to hold.  This is in line with what is known in the literature, as found for example in \cite{Bartz,Dao,Dong,Monteiro}.}}
\end{remark}

\section{Numerical Study}\label{sec:Numerical}

\noindent In this section, we consider the weighted Lasso minimization problem that is studied in the literature, such as \cite{Candes, Giselsson1, Monteiro}.  The problem is given by
\begin{eqnarray}\label{min:weightedLasso}
\min_{u \in \Re^n} f(u) + g(u),
\end{eqnarray}
where $f(u) := \frac{1}{2} \| Cu - b \|^2$ and $g(u) := \| Wu \|_1$ for every $u \in \Re^n$.  Here, $C \in \Re^{m \times n}$ and $b \in \Re^m$.  $C$ is a sparse data matrix with nonzero entries to zero entries in an average ratio of $1:20$ per row.  Each component of $b$ and each nonzero element of $C$ is drawn from the Gaussian distribution with zero mean and unit variance, while $W \in \Re^{n \times n}$ is a diagonal matrix with positive diagonal elements drawn from the uniform distribution on the interval $[0,1]$.  This setup is inspired by \cite{Giselsson1, Monteiro}.  It can be seen easily that $f$ defined above is a $\alpha$-strongly convex function on $\Re^n$, where $\alpha = \lambda_{\min}(C^T C)$ is the minimum eigenvalue of $C^T C$.  Hence, $\nabla f$ is maximal $\alpha$-strongly monotone.  Furthermore, $f$ is differentiable and its gradient is $\kappa$-Lipschitz continuous on $\Re^n$, where $\kappa = \lambda_{\max}(C^T C)$, which is the maximum eigenvalue of $C^T C$.  Clearly, $g$ defined above is a convex function on $\Re^n$.

\vspace{10pt}

\noindent We consider solving (\ref{min:weightedLasso}) with $f$ and $g$ defined in the previous paragraph using the relaxed PR splitting method (\ref{RelaxedPR}) on the monotone inclusion problem (\ref{MIP}) with 
\begin{eqnarray}\label{instanceAB}
A = \nabla f - \frac{\alpha}{2}I, \quad B = \partial g + \frac{\alpha}{2}I.
\end{eqnarray}
We have that $A$ and $B$ are maximal $\beta$-strongly monotone, where $\beta = \alpha/2$.

\vspace{10pt}

\noindent We first verify the convergence of iterates generated using (\ref{RelaxedPR}) with $A$ and $B$ given  by (\ref{instanceAB}) for $\theta \in (2 + \beta, 2 + \beta + \min\{\beta, 1/\beta\})$.  Convergence of iterates in this range of $\theta$ is predicted by our theory, in particular, Corollary \ref{cor:convergence}.  We consider a random instance of $(C, W, b)$ with $m = 300$ and $n = 200$, and set the initial iterate $x_0$ to be $(1, \ldots, 1)^T$.  We run the algorithm for $\theta$ varying from $1$ to $2 + \beta + \min \{ \beta, 1/\beta \}$, with the algorithm terminating at the $k^{th}$ iteration when $\| x_k - x_{k-1} \| \leq 10^{-5}$.  Our results are given in Figure \ref{figure1}. For this instance of $(C, W, b)$, $\beta$ is given by $0.1896$.  We see from Figure \ref{figure1} that the number of iterations performed before termination decreases as $\theta$ increases from 1, reaches its minimum around $\theta = 2.25$, which is close to $\theta = 2 + \beta$, and increases again.  Our results indicate that convergence using (\ref{RelaxedPR}) occurs as predicted by Corollary \ref{cor:convergence}, when $\theta$ lies between $2 + \beta$ and $2 + \beta + \min \{ \beta, 1/\beta\}$.  Indeed, when we set $\theta = 2 + \beta + \min \{\beta, 1/\beta\} + 0.05$, the algorithm still does not terminate when the number of iterations has reached $1000$. 

\begin{figure}[htpb]
\centering
\includegraphics[width=14cm]{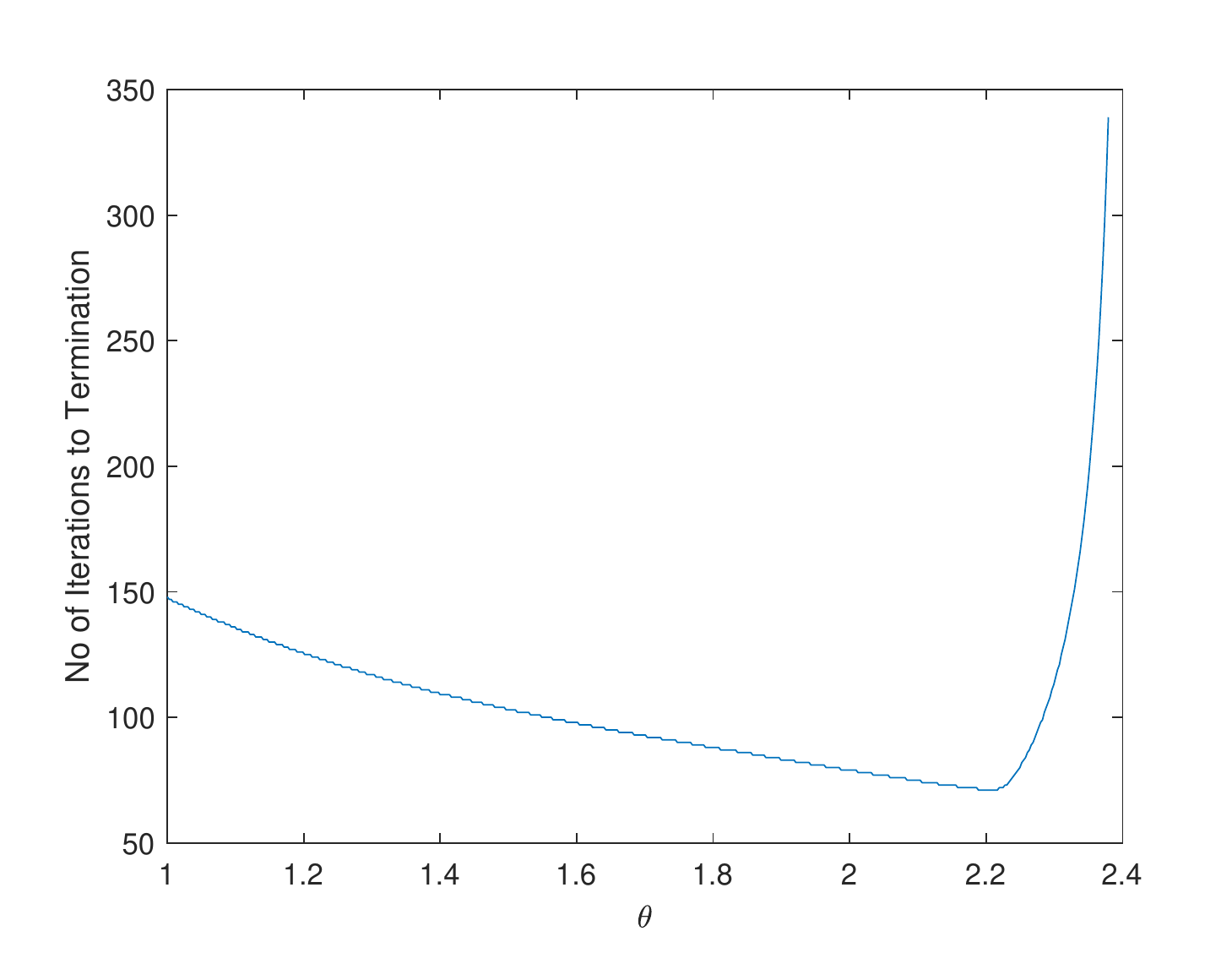}
\centering
\caption{Graph showing how the number of iterations before termination, upon applying (\ref{RelaxedPR}), varies with $\theta$.} \label{figure1}
\end{figure}  

\vspace{10pt}

\noindent Next, we test the validity of Assumption \ref{ass:Si} on the weighted Lasso minimization problem (\ref{min:weightedLasso}), by applying (\ref{RelaxedPR}) on (\ref{MIP}) with $A, B$ given by (\ref{instanceAB}).  In our numerical experiments, we set $n = m$ to be $200$ and $400$, and $x_0$ to be $(1, \ldots, 1)^T$ and $(0, -1, \ldots, -1)^T$.  For each of the four scenarios, we generate $(C, W)$ randomly $100$ times, and run the algorithm for $800$ iterations on those instances of $(C,W)$ with $\alpha = \lambda_{{\textcolor{black}{\min}}}(C^TC) > 0$, that is, $\beta > 0$.  We called these instances ``acceptable".  We set $\theta$ to be always equal to $2 + \beta + \frac{1}{2} \min \{ \beta, 1/\beta\}$.

\vspace{10pt}

\noindent  In order to test the validity of Assumption \ref{ass:Si}, we need to know the optimal solution to (\ref{min:weightedLasso}).  We do this by setting $b$ to be zero.  Hence, the optimal solution to (\ref{min:weightedLasso}) is $u^\ast = 0$.  Therefore, $x^\ast = 0$, since $u^\ast = J_A(x^\ast)$, where $A$ is given in (\ref{instanceAB}).   Hence, to validate Assumption \ref{ass:Si}, we need to verify that for each iterate, $x_k$, we have $x_k^i \not= 0$, for $i = 1, \ldots, n$.  
%To test the validity of Assumption \ref{ass:iterationconvergencerate}, we find the maximum $| (\gamma_{800} - \gamma_{799})/\gamma_{800} |$, the minimum $| (\gamma_{800} - \gamma_{799})/\gamma_{800} |$, and the number of instances with $| (\gamma_{800} - \gamma_{799})/\gamma_{800} | < 0.001$ of the ``acceptable" instances.  We consider both $b = 0$ and $b$ randomly generated when testing the validity of Assumption \ref{ass:iterationconvergencerate}.  

\vspace{10pt}

\noindent {\textcolor{black}{Our results are shown in Table \ref{table1}.  These results give preliminary indication that Assumption \ref{ass:Si} should hold in practice, and is only a technical assumption needed to prove convergence using the relaxed PR splitting method (\ref{RelaxedPR}) to solve the monotone inclusion problem (\ref{MIP}) for $\theta$ beyond the range of $(0, 2 + \beta]$, and less than $2 + \beta + \min \{  \beta, 1/\beta \}$.
}}
\begin{table}[htpb]
\centering
\begin{tabular}{||l|c|c||}  \hline
    & {\textcolor{black}{$\#$ ``acceptable" instances}} & $\begin{array}{l}
    							{\textcolor{black}{\#\ {\rm{instances\ with}}}}  \\ 										{\textcolor{black}{{{\rm{Assumption\ \ref{ass:Si}\ satisfied}}}}}
    								\end{array}$
    								 \\ \hline \hline
$\begin{array}{l}
n = m = 200 \\
x_0 = (1, \ldots, 1)^T 
%\min ({\rm{cond}}(C)) = 224.0324
\end{array}$
&   98  & 98 \\ \hline \hline
$\begin{array}{l}
n = m = 400 \\
x_0 = (1, \ldots, 1)^T 
%\min ({\rm{cond}}(C)) = 495.8961
\end{array}$ 
&   100  & 100 \\ \hline \hline
$\begin{array}{l}
n = m = 200 \\
x_0 = (0, -1, \ldots, -1)^T 
%\min ({\rm{cond}}(C)) = 313.1908
\end{array}$
&   100  & 100 \\ \hline \hline
$\begin{array}{l}
n = m = 400 \\
x_0 = (0, -1, \ldots, -1)^T 
%\min ({\rm{cond}}(C)) = 437.7462
\end{array}$
&   100  & 100  \\ \hline
%{\mbox{$n = m = 400$ \newline
%$x_0 = (1, \ldots, 1)^T$ \newline
%$\min ({\rm{cond}}(C)) = 495.8961$}}&   100  & 100 \\ \hline
%$b$ random & $\max \displaystyle \left( \left| \frac{\gamma_{800} - \gamma_{799}}{\gamma_{800}} \right| \right) = 21.6001$ with $\gamma_{800} = 7.1717 \times 10^{14}$ \\ \cline{2-2} 
%$\min ({\rm{cond}}(C)) \geq 1000$ & $\min \displaystyle \left( \left| \frac{\gamma_{800} - \gamma_{799}}{\gamma_{800}} \right| \right) = 0$ with $\gamma_{800} = 1.5383$ \\ \cline{2-2}
%& $\#$ instances with $\displaystyle \left( \left| \frac{\gamma_{800} - \gamma_{799}}{\gamma_{800}} \right| < 0.001\right)$ = $46$ \\ \hline
\end{tabular}
\caption{{\textcolor{black}{For each scenario, $100$ instances with $b = 0$ are generated.}}}\label{table1}
\end{table}

\section{Conclusion}\label{sec:Conclusion}

In this paper, we consider the relaxed PR splitting method (\ref{RelaxedPR}) to solve the monotone inclusion problem (\ref{MIP}).  We consider $A, B: \Re^n \rightrightarrows \Re^n$ in (\ref{MIP}) to be maximal $\beta$-strongly monotone operators, with $\beta > 0$, in this paper.  We show for the first time that for $\theta \in (2 + \beta, 2 + \beta + \min \{ \beta, 1/\beta \})$, an accumulation point, $x^{\ast\ast}$, of iterates $\{ x_k \}$ generated using (\ref{RelaxedPR}) on (\ref{MIP}) has $J_A(x^{\ast\ast})$ that solves (\ref{MIP}).  As a consequence, if $A$ or $B$ is single-valued, then $\{ x_k \}$ converges to a limit point $x^\ast$, where $J_A(x^\ast)$ solves (\ref{MIP}).  These are shown under a technical assumption.  Note that for $n = 1$, not having this technical assumption leads to trivial consideration. Furthermore, for $\theta \in [2 + \beta, 2 + \beta + \min \{ \beta, 1/\beta \})$, we provide pointwise convergence rate and {\textcolor{black}{$R$}}-linear convergence rate results of $\{x_k\}$ in Theorem \ref{thm:bargammaast} and Corollaries \ref{cor:ALipschitz}, \ref{cor:BLipschitz}.  Through numerical experiments on the weighted Lasso minimization problem, we provide preliminary evidence to show that the technical assumption used in this paper is likely to hold in practice.

%\vspace{10pt}
%
%\noindent {\bf{\underline{Data Availability Statement}}}
%
%\vspace{10pt}
%
%\noindent The Matlab program files for the numerical study in the paper are available from the corresponding author on reasonable request.

\section*{Acknowledgement}
{\textcolor{black}{
I would like to thank the Associate Editor for handling the paper, and the two reviewers for their comments and questions that help to improve the paper.  Finally, I would like to thank Prof. R.D.C. Monteiro for introducing me to the relaxed Peaceman-Rachford splitting method.
}}

\appendix
\section{Appendix}\label{sec:Appendix}
{\textcolor{black}{
\begin{proposition}\label{prop:appendix0}
{\textcolor{black}{We have
\begin{eqnarray*}
\max_{x \geq 0, 0 \leq y^2 \leq L^2/(1 + \beta)^2 - 1 + 2y} - x^2 + xy 
\end{eqnarray*}
is less than or equal to $\frac{1}{4}\left(1 + \frac{L}{1 + \beta} \right)^2$.}}
\end{proposition}
\begin{proof}
We have an optimal solution $(x^\ast, y^\ast)$ to the above maximization problem satisfies:
\begin{eqnarray}\label{eq:appendix0KKT0}
\left( \begin{array}{c}
		2x^\ast - y^\ast \\
		-x^\ast
		\end{array} \right) +
\lambda_1 \left( \begin{array}{c}
					-1 \\
					0
				\end{array} \right) +
\lambda_2 \left( \begin{array}{c}
					0 \\
					-1
				\end{array} \right) +
\lambda_3 \left( \begin{array}{c}
					0 \\
					2y^\ast - 2
				\end{array} \right) = 0,
\end{eqnarray}
with
\begin{eqnarray}
& & \lambda_1 x^\ast = 0, \quad \lambda_1 \geq 0, \quad x^\ast \geq 0, \label{rel:appendix0KKT1} \\
& & \lambda_2 y^\ast = 0, \quad \lambda_2 \geq 0, \quad y^\ast \geq 0, \label{rel:appendix0KKT2} \\
& & \lambda_3 \left((y^\ast)^2 - 2y^\ast + 1 - \frac{L^2}{(1+\beta)^2} \right) = 0, \quad \lambda_3 \geq 0, \quad (y^\ast)^2 \leq \frac{L^2}{(1 + \beta)^2} - 1 + 2 y^\ast. \label{rel:appendix0KKT3}
\end{eqnarray}
We first observe if $x^\ast = 0$, then $y^\ast = 0$, and the proposition is proved.
Suppose $x^\ast > 0$.  Then, from (\ref{rel:appendix0KKT1}), we have $\lambda_1 = 0$ and this implies that $y^\ast = 2 x^\ast$ from the first ``row" of (\ref{eq:appendix0KKT0}).  Also, observe from the second ``row" in (\ref{eq:appendix0KKT0}) with $x^\ast > 0$ and $\lambda_2 \geq 0$ that we must have $\lambda_3 > 0$.  Hence from (\ref{rel:appendix0KKT3}), we get 
\begin{eqnarray*}
(y^\ast)^2 - 2y^\ast + 1 - \frac{L^2}{(1+\beta)^2} = 0.
\end{eqnarray*}
Therefore,
\begin{eqnarray*}
y^\ast = \frac{2 \pm \sqrt{4 - 4\left(1 - \frac{L^2}{(1+\beta)^2} \right)}}{2} = 1 \pm \frac{L}{1 + \beta}.
\end{eqnarray*}
Now, since $y^\ast = 2 x^\ast > 0$, from (\ref{rel:appendix0KKT2}), we get $\lambda_2 = 0$.  Therefore, from the ``second" row in (\ref{eq:appendix0KKT0}), we have $x^\ast = \lambda_3(2y^\ast - 2) > 0$.   We observe then that $y^\ast \not= 1 - \frac{L}{1 + \beta}$.  Hence, $y^\ast = 1 + \frac{L}{1 + \beta}$.  And $x^\ast = \frac{1}{2} y^\ast = \frac{1}{2} \left(1 + \frac{L}{1 + \beta} \right)$.  The proposition is hence proved.  
\end{proof}
}}

\begin{proposition}\label{prop:appendix}
{\textcolor{black}{For $L < 1 + \beta$, we have
\begin{eqnarray*}
\max_{x,y\geq 0} - y^2 + \frac{4}{1+\beta}y + 2xy + \left[\frac{L^2}{(1+\beta)^2}- 1\right]x^2 - \frac{4}{(1 + \beta)^2}
\end{eqnarray*}
is less than or equal to zero.}}
\end{proposition}
\begin{proof}
{\textcolor{black}{We have an optimal solution  $(x^\ast,y^\ast)$ to the above maximization problem satisfies:
\begin{eqnarray}\label{eq:appendixKKT0}
\left( \begin{array}{c}
{\textcolor{black}{-}}2y^\ast {\textcolor{black}{-}} 2\left(\frac{L^2}{(1+\beta)^2} - 1\right)x^\ast \\
2y^\ast {\textcolor{black}{-}} 2x^\ast {\textcolor{black}{-}} \frac{4}{1 + \beta}
\end{array} \right) +
\lambda_1 \left( \begin{array}{c}
				-1 \\
				0
				\end{array} \right) +
\lambda_2 \left( \begin{array}{c}
				0 \\
				-1
				\end{array} \right) = 0,
\end{eqnarray}
with
\begin{eqnarray}
& & \lambda_1 x^\ast = 0, \quad \lambda_1\geq 0, \quad x^\ast \geq 0, \label{rel:appendixKKT1}\\
& & \lambda_2 y^\ast = 0, \quad \lambda_2 \geq 0, \quad y^\ast \geq 0. \label{rel:appendixKKT2}
\end{eqnarray}
Suppose $x^\ast > 0$.  Then, from (\ref{rel:appendixKKT1}), we have $\lambda_1 = 0$, and hence from the first ``row" in (\ref{eq:appendixKKT0}), we get
\begin{eqnarray}\label{eq:appendixyast1}
y^\ast = \left(1 - \frac{L^2}{(1+\beta)^2} \right) x^\ast.
\end{eqnarray}
Given that $L < 1 + \beta$, from (\ref{eq:appendixyast1}), we have $y^\ast > 0$, which then implies by (\ref{rel:appendixKKT2}) that $\lambda_2 = 0$.  Hence, from the second ``row" in (\ref{eq:appendixKKT0}), we obtain
\begin{eqnarray}\label{eq:appendixyast2}
y^\ast = x^\ast + \frac{2}{1+\beta}.
\end{eqnarray}
By (\ref{eq:appendixyast1}) and (\ref{eq:appendixyast2}), we have
\begin{eqnarray*}
x^\ast = -\frac{2(1+\beta)}{L^2},
\end{eqnarray*}
which is less than zero.  This is a contradiction to $x^\ast > 0$.  Hence, $x^\ast = 0$.  Suppose $y^\ast = 0$. Then the optimal value to the maximization problem is $-4/(1+\beta)^2$.  Suppose $y^\ast > 0$.  Then $\lambda_2 = 0$, and we have from the second ``row" of (\ref{eq:appendixKKT0}), $y^\ast = 2/(1+\beta)$.  Together with $x^\ast = 0$, the objective function value of the problem is then equal to zero, and we are done.}}
\end{proof}


\begin{thebibliography}{99}
\bibitem{Bartz}
{\textcolor{black}{S. Bartz, M.N. Dao \& H.M. Phan, \emph{Conical averagedness and convergence analysis of fixed point algorithms}, Journal of Global Optimization, {\bf{82}}(2022), pp. 351-373.}}
\bibitem{Bauschke}
H.H. Bauschke \& P.L. Combettes, \emph{Convex Analysis and Monotone Operator Theory in Hilbert Spaces}, Springer, 2011.
\bibitem{Bauschke1}
H.H. Bauschke, J.Y. Bello Cruz, T.T.A. Nghia, H.M. Phan \& X. Wang, \emph{The rate of linear convergence of the Douglas-Rachford algorithm for subspaces is the cosine of the Friedrichs angle}, Journal of Approximation Theory, {\bf{185}}(2014), pp. 63-79.
\bibitem{Bauschke2}
H.H. Bauschke, J.Y. Bello Cruz, T.T.A. Nghia, H.M. Phan \& X. Wang, \emph{Optimal rates of linear convergence of relaxed alternating projections and generalized Douglas-Rachford methods for two subspaces}, Numerical Algorithms, {\bf{73}}(2016), pp. 33-76.
\bibitem{Bauschke3}
H.H. Bauschke, D. Noll \& H.M. Phan, \emph{Linear and strong convergence of algorithms involving averaged nonexpansive operators}, Journal of Mathematical Analysis and Applications, {\bf{421}}(2015), pp. 1-20.
\bibitem{Candes}
E.J. Candes, M.B. Wakin \& S. Boyd, \emph{Enhancing sparsity by reweighted $l_1$ minimization}, Journal of Fourier Analysis and Applications, {\bf{14}}(2008), pp. 877 - 905.
\bibitem{Combettes}
P.L. Combettes, \emph{Solving monotone inclusions via compositions of nonexpansive averaged operators}, Optimization, {\bf{53}}(2004), pp. 475-504.
\bibitem{Combettes1}
P.L. Combettes, \emph{Iterative construction of the resolvent of a sum of maximal monotone operators}, Journal of Convex Analysis, {\bf{16}}(2009), pp. 727-748.
\bibitem{Dao}
{\textcolor{black}{M.N. Dao \& H.M. Phan, \emph{Adaptive Douglas-Rachford splitting algorithm for the sum of two operators}, SIAM Journal on Optimization, {\bf{29}}(2019), pp. 2697-2724.}}
\bibitem{Davis}
D. Davis, \emph{Convergence rate analysis of the Forward-Douglas-Rachford splitting scheme}, SIAM Journal on Optimization, {\bf{25}}(2015), pp. 1760-1786.
\bibitem{Davis1}
D. Davis \& W. Yin, \emph{Convergence rate analysis of several splitting schemes}, in ``Splitting Methods in Communication, Imaging, Science, and Engineering" (R. Glowinski, S. Osher, W. Yin, eds), pp. 115-163, Scientific Computation book series, Springer, Cham., 2016.
\bibitem{Davis2}
D. Davis \& W. Yin, \emph{Faster convergence rates of relaxed Peaceman-Rachford and ADMM under regularity assumptions}, Mathematics of Operations Research, {\bf{42}}(2017), pp. 783-805.
\bibitem{Dong}
Y. Dong \& A. Fischer, \emph{A family of operator splitting methods revisited}, Nonlinear Analysis, {\bf{72}}(2010), pp. 4307-4315.
\bibitem{Eckstein}
J. Eckstein \& D.P. Bertsekas, \emph{On the Douglas-Rachford splitting method and the proximal point algorithm for maximal monotone operators}, Mathematical Programming, {\bf{55}}(1992), pp. 293-318.
\bibitem{Facchinei}
F. Facchinei \& J.-S. Pang, \emph{Finite-Dimensional Variational Inequalities and Complementarity Problem}, Volume II, Springer-Verlag, 2003.
\bibitem{Giselsson}
P. Giselsson, \emph{Tight global linear convergence rate bounds for Douglas-Rachford splitting}, Journal of Fixed Point Theory and Applications, {\bf{19}}(2017), pp. 2241–2270.
\bibitem{Giselsson1}
P. Giselsson \& S. Boyd, \emph{Linear convergence and metric selection for Douglas-Rachford splitting and ADMM}, IEEE Transactions on Automatic Control, {\bf{62}}(2017), pp. 532 - 544.
%\bibitem{Goncalves}
%M.L.N. Goncalves, J.G. Melo \& R.D.C. Monteiro, \emph{Improved pointwise iteration-complexity of a regularized ADMM and of a regularized non-Euclidean HPE framework}, SIAM Journal on Optimization, {\bf{27}}(2017), pp. 379-407.
\bibitem{He}
B. He \& X. Yuan, \emph{On the convergence rate of Douglas-Rachford operator splitting method}, Mathematical Programming, Series A, {\bf{153}}(2015), pp. 715-722.
\bibitem{John}
F. John, \emph{Extremum problems with inequalities as side conditions}, in ``Studies and Essays, Courant Anniversary Volume" (K. O. Friedrichs, O. E. Neugebauer and J. J. Stoker, eds), pp. 187 - 204, Wiley (Interscience), New York, 1948.
%\bibitem{Kolossoski}
%O. Kolossoski \& R.D.C. Monteiro, \emph{An accelerated non-Euclidean hybrid proximal extragradient-type algorithm for convex-concave saddle-point problems}, Optimization Methods and Software, {\bf{32}}(2017), pp. 1244-1272.
\bibitem{Lions}
P.-L. Lions \& B. Mercier, \emph{Splitting algorithms for the sum of two nonlinear operators}, SIAM Journal on Numerical Analysis, {\bf{16}}(1979), pp. 964-979.
\bibitem{Mangasarian}
O.L. Mangasarian \& S. Fromovitz, \emph{The Fritz John necessary optimality conditions in the presence of equality and inequality constraints}, Journal of Mathematical Analysis and Applications, {\bf{17}}(1967), pp. 37 - 47.
\bibitem{Monteiro}
R.D.C. Monteiro \& C.-K. Sim, \emph{Complexity of the relaxed Peaceman-Rachford splitting method for the sum of two maximal strongly monotone operators,} Computational Optimization and Applications, {\bf{70}}(2018), pp. 763-790.
\bibitem{Sim}
{\textcolor{black}{C.-K. Sim, \emph{A note on convergence using the relaxed Peaceman-Rachford splitting method on the sum of two maximal strongly monotone operators,} unpublished manuscript, August 2020 (https://drive.google.com/file/d/1Yf5rn21mhof4j3u8qQRlFVzONfbtn6jV/view).}}
%\bibitem{Solodov}
%M.V. Solodov \& B.F. Svaiter, \emph{An inexact hybrid generalized proximal point algorithm and some new results on the theory of Bregman functions,} Mathematics of Operations Research, {\bf{25}}(2000), pp. 214-230.
%\bibitem{Wagner}
%H. M. Wagner \& T. M. Whitin, \emph{Dynamic version of the economic lot size model,} Management Science, {\bf{5}}(1958), pp. 89-96.
%\bibitem{Rockafellar}
%R. T. Rockafellar, \emph{Convex Analysis,} Princeton Landmarks in Mathematics, Princeton University Press, 1970.
\end{thebibliography}
\end{document}